\documentclass[12pt]{article}

\usepackage[no-biblatex,no-single-cal]{mypreambles}

\usepackage[matrix,arrow,tips,curve]{xy}
\usepackage{multicol}
\usepackage[ruled,vlined,longend]{algorithm2e}
\usepackage[sort]{cite}
\usepackage{graphicx}
\usepackage{doi}  

\usetikzlibrary{positioning}
\definecolor{bpdgray}{RGB}{240,240,240}
\definecolor{BPDblue}{RGB}{0,102,204}
\definecolor{coBPDgreen}{RGB}{34,139,34}
\definecolor{pipepurple}{RGB}{192,57,43}

\tikzset{
  bpdcell/.style={draw=black, line width=0.8pt},
  bpdline/.style={line width=1.8pt, rounded corners},
}

\newcommand{\BPDStart}[2]{%
  \begin{scope}[x=2em,y=2em, yscale=-1, shift={(0,#1)}]
}

\newcommand{\BPDEnd}{%
  \end{scope}
}

\newcommand{\BPDPlace}[3]{%
  \begin{scope}[shift={(#2-1,#1-1)}]
    #3
  \end{scope}%
}

\newcommand{\blank}{%
  \draw[bpdcell] (0,0) rectangle (1,1);
}

\newcommand{\+}{%
  \draw[bpdcell] (0,0) rectangle (1,1);
  \draw[BPDblue,bpdline] (0.5,0) -- (0.5,1);  
  \draw[BPDblue,bpdline] (0,0.5) -- (1,0.5);  
}

\newcommand{\hor}{%
  \draw[bpdcell] (0,0) rectangle (1,1);
  \draw[BPDblue,bpdline] (0,0.5) -- (1,0.5);
}

\newcommand{\ver}{%
  \draw[bpdcell] (0,0) rectangle (1,1);
  \draw[BPDblue,bpdline] (0.5,0) -- (0.5,1);
}

\newcommand{\NW}{
  \draw[bpdcell] (0,0) rectangle (1,1);
  \draw[BPDblue,bpdline] (0.5,0) -- (0.5,0.5) -- (0,0.5);
}

\newcommand{\SE}{
  \draw[bpdcell] (0,0) rectangle (1,1);
  \draw[BPDblue,bpdline] (0.5,1) -- (0.5,0.5) -- (1,0.5);
}

\tikzset{
  pipecell/.style={draw=black, line width=0.8pt},
  pipeline/.style={line width=1.8pt, rounded corners, color=pipepurple},
  pipelineghost/.style={pipeline, opacity=0.35}
}

\newcommand{\PipeStart}[1]{%
  \begin{scope}[x=2em,y=2em,yscale=-1,shift={(0,#1)}]
}

\newcommand{\PipeEnd}{%
  \end{scope}
}

\newcommand{\PipePlace}[3]{%
  \begin{scope}[shift={(#2-1,#1-1)}]
    #3
  \end{scope}
}

\newcommand{\pipeElbow}{%
  \draw[pipecell] (0,0) rectangle (1,1);
  \draw[pipeline] (0.5,0) .. controls (0.5,0.35) and (0.35,0.5) .. (0,0.5);
  \draw[pipelineghost] (1,0.5) .. controls (0.65,0.5) and (0.5,0.65) .. (0.5,1);
}

\newcommand{\pipeCross}{%
  \draw[pipecell] (0,0) rectangle (1,1);
  \draw[pipeline] (0.5,0) -- (0.5,1);
  \draw[pipeline] (0,0.5) -- (1,0.5);
}

\newcommand{\pipeBump}{%
    \draw[pipecell] (0,0) rectangle (1,1);
    \draw[pipeline] (0.5,0) .. controls (0.5,0.32) and (0.32,0.5) .. (0,0.5);
    \draw[pipeline] (0.5,1) .. controls (0.5,0.68) and (0.68,0.5) .. (1,0.5);
}

\newcommand{\pipeEmpty}{%
    \draw[pipecell] (0,0) rectangle (1,1);
}

\newcommand{\PipeLabels}[2]{%
  \begin{scope}
    \foreach[count=\i] \lab in {#1} {
      \node[font=\small, anchor=south, yshift=1pt] at (\i-0.5, 0) {\lab};
    }
    \foreach[count=\i] \lab in {#2} {
      \node[font=\small, anchor=east, xshift=-1pt] at (0, \i-0.5) {\lab};
    }
  \end{scope}
}

\newcommand{\pipetilebaseline}{0.2ex} 
\newcommand{\PipeTileInlineRot}[2][0]{%
  \tikz[baseline=\pipetilebaseline, x=1em, y=1em]{%
    \begin{scope}[shift={(.5,.5)}, rotate=#1, shift={(-.5,-.5)}]
      #2%
    \end{scope}%
  }%
}

\newcommand{\bumpangle}{90}  

\newcommand{\crosstile}{\PipeTileInlineRot[0]{\pipeCross}}
\newcommand{\bl}{\PipeTileInlineRot[0]{\pipeEmpty}}
\newcommand{\bt}{\PipeTileInlineRot[\bumpangle]{\pipeBump}}

\newcommand{\bpdblank}{\PipeTileInlineRot[0]{\blank}}
\newcommand{\bpdplus}{\PipeTileInlineRot[0]{\+}}
\newcommand{\bpdh}{\PipeTileInlineRot[0]{\hor}}
\newcommand{\bpdv}{\PipeTileInlineRot[0]{\ver}}

\newcommand{\bpdNW}{\PipeTileInlineRot[0]{\NW}}
\newcommand{\bpdSE}{\PipeTileInlineRot[0]{\SE}}

\providecommand{\bA}{\mathbb{A}}

\providecommand{\kk}{\mathbb{F}}

\providecommand{\cL}{\mathcal{L}}

\newcommand{\Xnk}{X_{n,k}}
\newcommand{\Rnk}{R_{n,k}}
\newcommand{\Pnk}{(\PP^{k-1})^n}
\newcommand{\Words}{[k]^n}
\newcommand{\Fubini}{\operatorname{Fubini}(n,k)}
\newcommand{\std}{\mathrm{std}}
\newcommand{\conv}{\mathrm{conv}}
\newcommand{\sinv}{\sigma^{-1}(w)}
\newcommand{\Mat}{\operatorname{Mat}_{n,k}}
\newcommand{\Nonzerocol}{\operatorname{Mat}^\circ_{n,k}}
\newcommand{\FR}{\operatorname{FR}_{n,k}}
\newcommand{\PM}{\operatorname{PM}}

\newtheorem{question}[theorem]{Question}

\crefname{axiom}{Axiom}{Axioms}
\Crefname{axiom}{Axiom}{Axioms}
\crefname{question}{Question}{Questions}
\Crefname{question}{Question}{Questions}
\crefname{conjecture}{Conjecture}{Conjectures}
\Crefname{conjecture}{Conjecture}{Conjectures}

\title{The Grothendieck Group of \\ the Variety of Spanning Line Configurations}
\author{Michael Ruofan Zeng}
\date{}

\begin{document}

\maketitle

\begin{abstract}
We study the Grothendieck group of the variety $X_{n,k}$ of spanning line configurations introduced by Pawlowski--Rhoades \cite{PawlowskiRhoades2019} as a geometric model for the generalized coinvariant algebra $R_{n,k}$.  Our first result is a localization statement in $K$-theory for the complements of cell closures in smooth cellular varieties. Combining with the Fulton--Lascoux degeneracy loci formula, we prove that $K_0(X_{n,k})$ is canonically isomorphic to $R_{n,k}$, extending classical isomorphisms for the flag variety. We next identify the classes of the Pawlowski--Rhoades varieties in $K_0((\PP^{k-1})^n)$ with Grothendieck polynomials associated to words $w \in [k]^n$.  Motivated by this identification, we develop models of classical and bumpless pipe dreams for general words, and we show that Schubert and Grothendieck polynomials of words are monomial-weight generating functions for these pipe dreams, extending the classical story from permutations to Fubini words.
\end{abstract}

\tableofcontents

\section{Introduction}
\label{sec:intro}

The \textit{coinvariant algebra} $R_n$ is the quotient ring
\[
R_n :=  \frac{\ZZ[x_1,x_2, \dots, x_n]}{(e_1, e_2, \dots, e_n)},
\]
where $e_i$ is the elementary symmetric polynomial of degree $i$. It has been a central object in commutative algebra, invariant theory, and algebraic combinatorics. In particular, it is a classical result of Borel that the integral cohomology ring of the flag variety $Fl_n(\CC)$ has the \textit{Borel presentation} \cite{Borel1953SurLC}
\begin{equation}\label{eq:Borel-presentation}
    H^\bullet(Fl_n(\CC);\ZZ) \cong R_n .
\end{equation}
The geometry of the complex algebraic variety $Fl_n(\CC)$, or more generally $Fl_n$ taken over any field $\kk$, is heavily intertwined with the algebra and combinatorics arising from $R_n$.

Since Borel's work, many other topological invariants of the flag variety have been found to be isomorphic to $R_n$, such as the Chow ring $CH^\bullet(Fl_n)$ of algebraic cycles and the Grothendieck group $K_0(Fl_n)$ of algebraic vector bundles on $Fl_n$ \cite{textchapter,Karoubi1978,LENART2005120}. Since $Fl_n$ has a cell decomposition into Schubert cells, the cycle class map gives an isomorphism $CH^\bullet(Fl_n) \cong H^\bullet(Fl_n(\CC);\ZZ)$. Demazure \cite{demazure1973invariants} showed that there is an isomorphism $K_0(Fl_n) \cong R_n$. Hence, we have a series of isomorphisms of ungraded rings 
\begin{equation}\label{eqn:isos-of-Fln}
    K_0(Fl_n) \cong CH^\bullet(Fl_n)\cong H^\bullet(Fl_n(\CC);\ZZ) \cong R_n.
\end{equation}
Special families of polynomials are present in these rings. \textit{Grothendieck polynomials} represent $K$-theoretic \textit{Schubert classes} and form an additive basis for $K_0(Fl_n)$. \textit{Schubert polynomials} represent the Schubert classes as an additive basis for the Chow ring and cohomology. Both families of polynomials have natural interpretations in terms of a combinatorial gadget called \textit{pipe dreams}. The study of these polynomials has been a fertile ground for results interlacing algebraic geometry and combinatorics; see, e.g., \cite{lascoux1982polynomes,BJS93,Lenart2000,LENART2005120,MONICAL_2020,MR4681291}.

A natural variation of the coinvariant algebra in the theory of symmetric functions, first studied by Haglund-Rhoades-Shimozono \cite{HaglundRhoadesShimozono2018}, is the \textit{generalized coinvariant algebra}
\begin{equation}
    \label{eqn:generalized-coinvariant-algebra}
    R_{n,k}:= \frac{\ZZ[x_1, x_2, \dots, x_n]}{(x_1^k, x_2^k, \dots, x_n^k, e_{n-k+1}, e_{n-k+2}, \dots, e_{n})},
\end{equation} 
where $1\le k\leq n$ are integers.  In analogy with $R_n$ and $Fl_n$, it is natural to ask whether there is a variety whose cohomology ring is isomorphic to $R_{n,k}$.  Pawlowski-Rhoades \cite{PawlowskiRhoades2019} answered this question by constructing such a variety.

Throughout this paper, let $\kk$ be a field of characteristic not equal to $2$.  Following \cite{PawlowskiRhoades2019}, a line configuration $\ell_\bullet$ of length $n$ in $\kk^k$ is an ordered $n$-tuple $(\ell_1, \ell_2, \dots, \ell_n)$, where each $\ell_i$ is a one-dimensional linear subspace of $\kk^k$, called a line. A \textit{spanning line configuration} is then a line configuration for which the linear span $\ell_1 + \ell_2 + \cdots + \ell_n$ equals all of $\kk^k$.

\begin{definition}
    The space of \textit{spanning line configurations} $X_{n,k}$ is $$X_{n,k} := \left\{\ell_\bullet :=(\ell_1, \dots, \ell_n)\mid \ell_1 + \cdots + \ell_n = \kk^k\right\}.$$
\end{definition}

The space $\Xnk$ has many interesting properties. It can be identified as an open subset of the product of projective spaces $(\PP^{k-1})^n$ by sending each line $\ell_i$ in $\ell_\bullet$ to the corresponding point in $\PP^{k-1}$. In addition, $\Xnk$ is \textit{cellular} with cells indexed by \textit{Fubini words}. Let $[k]:=\{1,2,\dots, k\}$. Then, a Fubini word  is $w = w_1w_2\dots w_n \in [k]^n$ such that each letter in $[k]$ appears at least once. The cells $C_w\subseteq \Xnk$ are called \textit{Pawlowski-Rhoades (PR) cells} and their closures $X_w :=\overline{C}_w$ are called \textit{Pawlowski-Rhoades (PR) varieties}. They are closely connected to the combinatorics of Fubini words \cite{PawlowskiRhoades2019, BilleyRyan2024}. The variety $\Xnk$ can be expressed as the open complement in $\Pnk$ of closures of cells corresponding to non-Fubini words. Pawlowski-Rhoades used this decomposition and Fulton's degeneracy loci formula for Schubert polynomials (\Cref{thm:Schubert-degeneracy-loci-formula}) to prove that there is a canonical isomorphism \begin{equation}
    H^\bullet(X_{n,k}(\CC);\ZZ)\cong \Rnk.
\end{equation} 
Thus $X_{n,k}$ provides a geometric incarnation of the generalized coinvariant algebra, just as $Fl_n$ does for $R_n$. See \Cref{sec:Xnk-cohomology} for a detailed description.

The first goals of this paper are to prove a complete analogue of the $Fl_n$ isomorphisms \eqref{eqn:isos-of-Fln} for the variety $X_{n,k}$ and to identify a natural analogue of the Schubert basis for $K_0(\Xnk)$. To do that, we establish the $K$-theoretic analogue of a key result for the cohomology of the open complement of cell closures which played an essential role in Pawlowski-Rhoades' proofs \cite[Theorem 4.5]{PawlowskiRhoades2019}. 

\begin{theorem}\label{thm:K0(X-Z)=K0(X)/I(Z)}
 Suppose $X$ is a smooth variety over $\kk$ that admits a cellular decomposition. Let $Z \subseteq  X$ be a union of cell closures. Let $I(Z)$ be the ideal in $K_0(X)$ generated by the fundamental classes of cell closures $[\cO_{\overline{C}}]$ for all cells $C \subseteq  Z$. The inclusion $Z \hookrightarrow X$ induces an isomorphism of rings \begin{equation} \label{eqn:K0(X-Z)=K0(X)/I(Z)}
     K_0(X-Z) \cong K_0(X)/I(Z). 
     \end{equation}
\end{theorem}

We were unable to locate a direct reference for this statement, so we give a proof of the above \Cref{thm:K0(X-Z)=K0(X)/I(Z)} in \Cref{subsec:key-lemmas-for-cellular}. By \eqref{eqn:K0(X-Z)=K0(X)/I(Z)}, it suffices to identify the classes of PR varieties $X_w$ where $w \in [k]^n$ is non-Fubini. Using the Fulton-Lascoux Degeneracy Loci Formula (\Cref{thm:Grothendieck-degeneracy-loci-formula}), we can express these classes in $K_0(\Pnk)$ in the form of Grothendieck polynomials. See \Cref{thm:grothendieck-poly-of-word} for details. Combining \Cref{thm:K0(X-Z)=K0(X)/I(Z)} and \Cref{thm:grothendieck-poly-of-word}, we can prove the main theorem of this paper. 

\begin{theorem} \label{thm:main}

Let $\Rnk$ be the generalized coinvariant algebra defined as above. There are isomorphisms of rings 
\begin{equation}
    K_0(X_{n,k}) \cong CH^\bullet(\Xnk) \cong H^\bullet(\Xnk(\CC);\ZZ) \cong \Rnk
\end{equation}
  over the integers. In particular, the identification $K_0(\Xnk)\cong \Rnk$ is canonical in terms of $K$-theoretic Chern classes of tautological line bundles.
\end{theorem}

As a result, $\Xnk$ joins the list of varieties whose Grothendieck group is abstractly isomorphic to the Chow ring as rings over the integers.  Other examples include projective spaces, Grassmannians, flag varieties, and wonderful varieties in the sense of De Concini-Procesi \cite{ConciniProcesi95Wonderful,LLPP24WonderfulVarieties}.  See \Cref{sec:open-problems} for a more detailed discussion.  In general, for a smooth projective variety, the rationalized Chern character gives an isomorphism $ch_\QQ: K_0(X)_\QQ \to CH^\bullet(X)_\QQ$. However, there need not exist any (ungraded) isomorphism between $K_0(X)$ and $CH^\bullet(X)$ over the integers, even in the case that $X$ is smooth, cellular, and projective. For example, the projective bundle $\PP(\cO_{\PP^2}\oplus\cO_{\PP^2}(2))$ over $\PP^2$ is cellular, but its Chow ring and $K_0$-ring are non-isomorphic over the integers. This motivates us to ask the following question.

\begin{question}\label{qs:when-are-K0-CH-iso}
    What are the necessary and sufficient conditions on a scheme $X$ such that there is an abstract ring isomorphism $K_0(X) \cong CH^\bullet(X)$ over the integers?
\end{question}

The classes of the PR varieties $[\cO_{X_w}]$ are interesting in their own right.  To state the results, we set up some notation \textit{a priori} and postpone the precise discussion until later sections. For a word $w\in[k]^n$, let $\mathfrak{G}_{\operatorname{std}(\conv(w))}$ be the Grothendieck polynomial of the \textit{standardized permutation} (\Cref{def:standardization}) of the \textit{convexification} of $w$ (\Cref{def:convexification}), and let $\sinv \in S_n$ be the \textit{associated permutation} (\Cref{def:convexification}).

\begin{definition}
   \label{def:grothendieck-poly-of-word}
    We define the \textit{Grothendieck polynomial} of a word $w\in [k]^n$ to be \begin{equation}
        \label{eqn:grothendieck-poly-of-word}
        \mathfrak{G}_w(x_1, \dots, x_n) := \sigma^{-1}(w)\cdot \mathfrak{G}_{\operatorname{std}(\conv(w))}(x_1, \dots, x_n)
    \end{equation}
    where $\sinv \in S_n$ acts by permuting the variables.
\end{definition}
Fix the quotient presentation of the Grothendieck group to be \begin{equation}
\label{eqn:K0(Pnk)} K_0(\Pnk) \cong S_{n,k}:=\mathbb{Z}[x_1, x_2, \dots, x_n]/(x_1^k, x_2^k, \dots, x_n^k).
\end{equation} Note that $R_{n,k}$ is a further quotient of $S_{n,k}$ by the ideal $(e_{n-k+1}, \dots, e_n)$. Then, the following theorem holds.

\begin{theorem}\label{thm:grothendieck-poly-of-word} For a word $w\in[k]^n$, the class $[\cO_{X_w}]$ in $K_0(\Pnk)$ maps to the Grothendieck polynomial $\mathfrak{G}_w(x_1, \dots, x_n)$
under the fixed isomorphism $K_0(\Pnk) \cong S_{n,k}$.
\end{theorem}

Motivated by Pawlowski-Rhoades' definition of the \textit{Schubert polynomial of a word} and \Cref{def:grothendieck-poly-of-word}, we propose natural notions of \textit{classical} and \textit{bumpless pipe dreams} for words in $\Words$ in \Cref{sec:PD-for-words}. Below are all reduced pipe dreams for the Fubini word $21231\in [3]^5$. Notice how these pipe dreams fit inside a $5\times 3$ box, reflecting $n = 5$ and $k = 3$. 

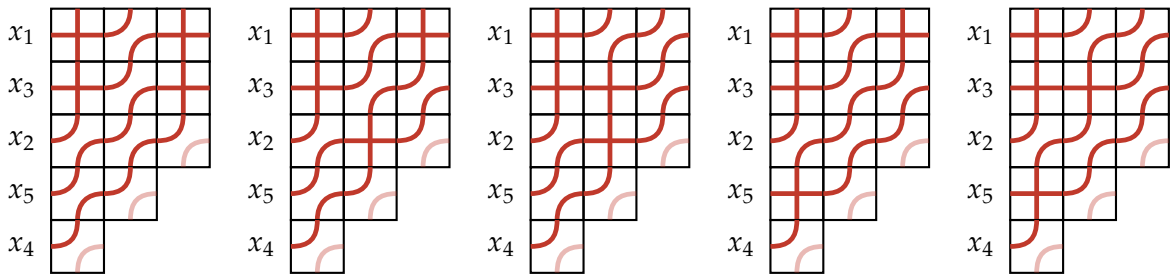
\begin{footnotesize}
    \[\tikzpicture
  \PipeStart{5}
  \PipePlace{1}{1}{\pipeCross}
  \PipePlace{1}{2}{\pipeBump}
  \PipePlace{1}{3}{\pipeCross}
  \PipePlace{2}{1}{\pipeCross}
  \PipePlace{2}{2}{\pipeBump}
  \PipePlace{2}{3}{\pipeCross}
  \PipePlace{3}{1}{\pipeBump}
  \PipePlace{3}{2}{\pipeBump}
  \PipePlace{3}{3}{\pipeElbow}
  \PipePlace{4}{1}{\pipeBump}
  \PipePlace{4}{2}{\pipeElbow}
  \PipePlace{5}{1}{\pipeElbow}
  \PipeEnd
  \begin{scope}[x=2em,y=2em, yscale=-1, shift={(-6,5)}]
    \node[right] at (5,0.5) {\small $x_1$};
    \node[right] at (5,1.5) {\small $x_3$};
    \node[right] at (5,2.5) {\small $x_2$};
    \node[right] at (5,3.5) {\small $x_5$};
    \node[right] at (5,4.5) {\small $x_4$};
  \end{scope}
\endtikzpicture \quad
\tikzpicture
  \PipeStart{5}
  \PipePlace{1}{1}{\pipeCross}
  \PipePlace{1}{2}{\pipeBump}
  \PipePlace{1}{3}{\pipeCross}
  \PipePlace{2}{1}{\pipeCross}
  \PipePlace{2}{2}{\pipeBump}
  \PipePlace{2}{3}{\pipeBump}
  \PipePlace{3}{1}{\pipeBump}
  \PipePlace{3}{2}{\pipeCross}
  \PipePlace{3}{3}{\pipeElbow}
  \PipePlace{4}{1}{\pipeBump}
  \PipePlace{4}{2}{\pipeElbow}
  \PipePlace{5}{1}{\pipeElbow}
  \PipeEnd
  \begin{scope}[x=2em,y=2em, yscale=-1, shift={(-6,5)}]
    \node[right] at (5,0.5) {\small $x_1$};
    \node[right] at (5,1.5) {\small $x_3$};
    \node[right] at (5,2.5) {\small $x_2$};
    \node[right] at (5,3.5) {\small $x_5$};
    \node[right] at (5,4.5) {\small $x_4$};
  \end{scope}
\endtikzpicture
\quad
\tikzpicture
  \PipeStart{5}
  \PipePlace{1}{1}{\pipeCross}
  \PipePlace{1}{2}{\pipeBump}
  \PipePlace{1}{3}{\pipeBump}
  \PipePlace{2}{1}{\pipeCross}
  \PipePlace{2}{2}{\pipeCross}
  \PipePlace{2}{3}{\pipeBump}
  \PipePlace{3}{1}{\pipeBump}
  \PipePlace{3}{2}{\pipeCross}
  \PipePlace{3}{3}{\pipeElbow}
  \PipePlace{4}{1}{\pipeBump}
  \PipePlace{4}{2}{\pipeElbow}
  \PipePlace{5}{1}{\pipeElbow}
  \PipeEnd
  \begin{scope}[x=2em,y=2em, yscale=-1, shift={(-6,5)}]
    \node[right] at (5,0.5) {\small $x_1$};
    \node[right] at (5,1.5) {\small $x_3$};
    \node[right] at (5,2.5) {\small $x_2$};
    \node[right] at (5,3.5) {\small $x_5$};
    \node[right] at (5,4.5) {\small $x_4$};
  \end{scope}
\endtikzpicture
\quad 
\tikzpicture
  \PipeStart{5}
  \PipePlace{1}{1}{\pipeCross}
  \PipePlace{1}{2}{\pipeBump}
  \PipePlace{1}{3}{\pipeCross}
  \PipePlace{2}{1}{\pipeCross}
  \PipePlace{2}{2}{\pipeBump}
  \PipePlace{2}{3}{\pipeBump}
  \PipePlace{3}{1}{\pipeBump}
  \PipePlace{3}{2}{\pipeBump}
  \PipePlace{3}{3}{\pipeElbow}
  \PipePlace{4}{1}{\pipeCross}
  \PipePlace{4}{2}{\pipeElbow}
  \PipePlace{5}{1}{\pipeElbow}
  \PipeEnd
  \begin{scope}[x=2em,y=2em, yscale=-1, shift={(-6,5)}]
    \node[right] at (5,0.5) {\small $x_1$};
    \node[right] at (5,1.5) {\small $x_3$};
    \node[right] at (5,2.5) {\small $x_2$};
    \node[right] at (5,3.5) {\small $x_5$};
    \node[right] at (5,4.5) {\small $x_4$};
  \end{scope}
\endtikzpicture
\quad 
\tikzpicture
  \PipeStart{5}
  \PipePlace{1}{1}{\pipeCross}
  \PipePlace{1}{2}{\pipeBump}
  \PipePlace{1}{3}{\pipeBump}
  \PipePlace{2}{1}{\pipeCross}
  \PipePlace{2}{2}{\pipeCross}
  \PipePlace{2}{3}{\pipeBump}
  \PipePlace{3}{1}{\pipeBump}
  \PipePlace{3}{2}{\pipeBump}
  \PipePlace{3}{3}{\pipeElbow}
  \PipePlace{4}{1}{\pipeCross}
  \PipePlace{4}{2}{\pipeElbow}
  \PipePlace{5}{1}{\pipeElbow}
  \PipeEnd
  \begin{scope}[x=2em,y=2em, yscale=-1, shift={(-6,5)}]
    \node[right] at (5,0.5) {\small $x_1$};
    \node[right] at (5,1.5) {\small $x_3$};
    \node[right] at (5,2.5) {\small $x_2$};
    \node[right] at (5,3.5) {\small $x_5$};
    \node[right] at (5,4.5) {\small $x_4$};
  \end{scope}
\endtikzpicture
\]
\end{footnotesize}
Next, we enumerate all bumpless pipe dreams for $21231$. The $x_i$ variables on the left dictate how monomial weights for these pipe dreams are computed.
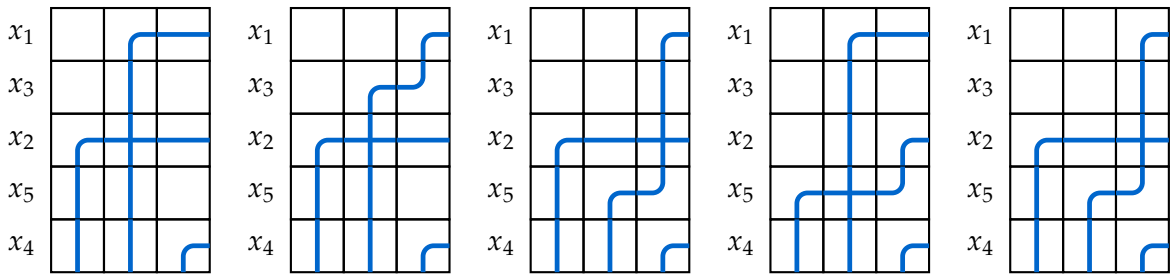
\begin{footnotesize}
\[
\tikzpicture
  \BPDStart{5}{3}
  \BPDPlace{1}{1}{\blank}
  \BPDPlace{1}{2}{\SE}
  \BPDPlace{1}{3}{\hor}
  \BPDPlace{2}{1}{\blank}
  \BPDPlace{2}{2}{\ver}
  \BPDPlace{2}{3}{\blank}
  \BPDPlace{3}{1}{\SE}
  \BPDPlace{3}{2}{\+}
  \BPDPlace{3}{3}{\hor}
  \BPDPlace{4}{1}{\ver}
  \BPDPlace{4}{2}{\ver}
  \BPDPlace{4}{3}{\blank}
  \BPDPlace{5}{1}{\ver}
  \BPDPlace{5}{2}{\ver}
  \BPDPlace{5}{3}{\SE}
  \BPDEnd
  \begin{scope}[x=2em,y=2em, yscale=-1, shift={(-6,5)}]
    \node[right] at (5,0.5) {\small $x_1$};
    \node[right] at (5,1.5) {\small $x_3$};
    \node[right] at (5,2.5) {\small $x_2$};
    \node[right] at (5,3.5) {\small $x_5$};
    \node[right] at (5,4.5) {\small $x_4$};
  \end{scope}
\endtikzpicture
\quad 
\tikzpicture
  \BPDStart{5}{5}
  \BPDPlace{1}{1}{\blank}
  \BPDPlace{1}{2}{\blank}
  \BPDPlace{1}{3}{\SE}
  \BPDPlace{2}{1}{\blank}
  \BPDPlace{2}{2}{\SE}
  \BPDPlace{2}{3}{\NW}
  \BPDPlace{3}{1}{\SE}
  \BPDPlace{3}{2}{\+}
  \BPDPlace{3}{3}{\hor}
  \BPDPlace{4}{1}{\ver}
  \BPDPlace{4}{2}{\ver}
  \BPDPlace{4}{3}{\blank}
  \BPDPlace{5}{1}{\ver}
  \BPDPlace{5}{2}{\ver}
  \BPDPlace{5}{3}{\SE}
  \BPDEnd
 \begin{scope}[x=2em,y=2em, yscale=-1, shift={(-6,5)}]
    \node[right] at (5,0.5) {\small $x_1$};
    \node[right] at (5,1.5) {\small $x_3$};
    \node[right] at (5,2.5) {\small $x_2$};
    \node[right] at (5,3.5) {\small $x_5$};
    \node[right] at (5,4.5) {\small $x_4$};
  \end{scope}
\endtikzpicture \quad
\tikzpicture
  \BPDStart{5}{5}
  \BPDPlace{1}{1}{\blank}
  \BPDPlace{1}{2}{\blank}
  \BPDPlace{1}{3}{\SE}
  \BPDPlace{2}{1}{\blank}
  \BPDPlace{2}{2}{\blank}
  \BPDPlace{2}{3}{\ver}
  \BPDPlace{3}{1}{\SE}
  \BPDPlace{3}{2}{\hor}
  \BPDPlace{3}{3}{\+}
  \BPDPlace{4}{1}{\ver}
  \BPDPlace{4}{2}{\SE}
  \BPDPlace{4}{3}{\NW}
  \BPDPlace{5}{1}{\ver}
  \BPDPlace{5}{2}{\ver}
  \BPDPlace{5}{3}{\SE}
  \BPDEnd
   \begin{scope}[x=2em,y=2em, yscale=-1, shift={(-6,5)}]
    \node[right] at (5,0.5) {\small $x_1$};
    \node[right] at (5,1.5) {\small $x_3$};
    \node[right] at (5,2.5) {\small $x_2$};
    \node[right] at (5,3.5) {\small $x_5$};
    \node[right] at (5,4.5) {\small $x_4$};
  \end{scope}
\endtikzpicture \quad
\tikzpicture
  \BPDStart{5}{5}
  \BPDPlace{1}{1}{\blank}
  \BPDPlace{1}{2}{\SE}
  \BPDPlace{1}{3}{\hor}
  \BPDPlace{2}{1}{\blank}
  \BPDPlace{2}{2}{\ver}
  \BPDPlace{2}{3}{\blank}
  \BPDPlace{3}{1}{\blank}
  \BPDPlace{3}{2}{\ver}
  \BPDPlace{3}{3}{\SE}
  \BPDPlace{4}{1}{\SE}
  \BPDPlace{4}{2}{\+}
  \BPDPlace{4}{3}{\NW}
  \BPDPlace{5}{1}{\ver}
  \BPDPlace{5}{2}{\ver}
  \BPDPlace{5}{3}{\SE}
  \BPDEnd
   \begin{scope}[x=2em,y=2em, yscale=-1, shift={(-6,5)}]
    \node[right] at (5,0.5) {\small $x_1$};
    \node[right] at (5,1.5) {\small $x_3$};
    \node[right] at (5,2.5) {\small $x_2$};
    \node[right] at (5,3.5) {\small $x_5$};
    \node[right] at (5,4.5) {\small $x_4$};
  \end{scope}
\endtikzpicture \quad
\tikzpicture
  \BPDStart{5}{5}
  \BPDPlace{1}{1}{\blank}
  \BPDPlace{1}{2}{\blank}
  \BPDPlace{1}{3}{\SE}
  \BPDPlace{2}{1}{\blank}
  \BPDPlace{2}{2}{\blank}
  \BPDPlace{2}{3}{\ver}
  \BPDPlace{3}{1}{\SE}
  \BPDPlace{3}{2}{\hor}
  \BPDPlace{3}{3}{\+}
  \BPDPlace{4}{1}{\ver}
  \BPDPlace{4}{2}{\SE}
  \BPDPlace{4}{3}{\NW}
  \BPDPlace{5}{1}{\ver}
  \BPDPlace{5}{2}{\ver}
  \BPDPlace{5}{3}{\SE}
  \BPDEnd
  \begin{scope}[x=2em,y=2em, yscale=-1, shift={(-6,5)}]
    \node[right] at (5,0.5) {\small $x_1$};
    \node[right] at (5,1.5) {\small $x_3$};
    \node[right] at (5,2.5) {\small $x_2$};
    \node[right] at (5,3.5) {\small $x_5$};
    \node[right] at (5,4.5) {\small $x_4$};
  \end{scope}
\endtikzpicture
\]
\end{footnotesize}
Our definition of bumpless pipe dreams of a word is compatible with the diagram of a word, as in Billey-Ryan \cite{BilleyRyan2024}. As a straightforward consequence of these definitions, we prove that Schubert polynomials of words are monomial-weight generating functions for reduced classical and bumpless pipe dreams. We also prove that Grothendieck polynomials of words are monomial-weight generating functions for non-reduced classical and bumpless pipe dreams. These results are complete analogues of classical results for Schubert and Grothendieck polynomials for permutations.

The structure of the paper is as follows.  In \Cref{sec:background}, we collect the necessary background on the Grothendieck group, the Chow ring, Chern classes, cellular varieties, Schubert and Grothendieck polynomials, and degeneracy loci, and we review classical and bumpless pipe dreams for permutations.  In \Cref{sec:Xnk-cohomology}, we recall the Pawlowski--Rhoades cell decomposition of $X_{n,k}$ and their presentation of $H^\bullet(X_{n,k}(\CC);\ZZ)$ in terms of Schubert polynomials of words.  In \Cref{sec:K0-of-Xnk}, we establish the key $K$-theoretic lemmas for cellular varieties, identify the classes $[\cO_{X_w}]$ with Grothendieck polynomials of words, and prove \Cref{thm:main}.  In \Cref{sec:PD-for-words}, we develop pipe dream and bumpless pipe dream models for words and relate them to the Schubert and Grothendieck bases of $R_{n,k}$. In \Cref{sec:open-problems}, we discuss several directions for further work.

\section{Background}\label{sec:background}

\subsection{Algebro-geometric Prerequisites}

 \label{subsec:grothendieck-group}

We provide the necessary algebro-geometric background in this section. We begin by summarizing some facts about two kinds of Grothendieck groups of a scheme from \cite[\S 2]{weibel2013k}. We use capital letters like $X$ for schemes, calligraphic letters like $\mathcal{E}$ to denote sheaves and vector bundles, square brackets like $[\mathcal{E}]$ to mean an equivalence class, and bold font letters like $\mathbf E$ to denote categories. An \textit{exact category} (in the sense of Quillen) is an additive category $\mathbf{E}$ with a well-defined notion of short exact sequences. Let $\mathbf{E}$ be an exact category with a set of isomorphism classes $\operatorname{Iso}\mathbf{E}$. Then, the \textit{Grothendieck group} $K_0(\mathbf{E})$ is the free abelian group $\ZZ\cdot \operatorname{Iso}\mathbf{E}$ modulo the relations
\[
  [B]=[A]+[C]\quad\text{whenever } A\to B\to C\text{ is a short exact sequence in }\mathbf{E}.
\]
An \textit{abelian category} is a category enriched over abelian groups, with a distinguished zero object, where kernels and cokernels are well-defined for all morphisms. Abelian categories are special cases of exact categories. 

Let $\kk$ be any base field. Let $X$ be an $\kk$-scheme of finite type, and let $\mathbf{Coh}_X$ be the abelian category of coherent $\cO_X$-modules under direct sum. Let $\mathbf{VB}_X$ be the exact subcategory of $\operatorname{Coh}_X$ consisting of locally free sheaves of finite rank (equivalently, \textit{algebraic vector bundles}). Note that $\mathbf{VB}_X$ is not abelian in general, since kernels and cokernels of maps of vector bundles need not be locally free. There are two Grothendieck groups naturally associated to the scheme $X$, which we denote by $G_0(X)$ and $K_0(X)$ following standard notation. 

\begin{definition}
  Let $X$ be an $\kk$-scheme. The \textit{Grothendieck group of coherent sheaves} on $X$ is \[G_0(X) := K_0(\mathbf{Coh}_X).\]
  The \textit{Grothendieck group of vector bundles} on $X$ is \[K_0(X):= K_0(\mathbf{VB}_X).\]
\end{definition}

The natural inclusion $\mathbf{VB}_X\to \mathbf{Coh}_X$ induces a map on the abelian groups $K_0(X)\to G_0(X)$ called the Cartan homomorphism. The tensor product $\otimes$ of vector bundles endows $K_0(X)$ with the structure of a commutative unital ring. We write $[\mathcal L\otimes \mathcal M]$ as $[\mathcal L]\cdot [\mathcal M]$. The unit of $K_0(X)$ is the class of the trivial line bundle $[\cO_X]$, which we abbreviate as $1\in K_0(X)$.  Now assume $X$ to be smooth. Then, every coherent $\cO_X$-module has a finite resolution by algebraic vector bundles, so the Cartan homomorphism becomes an isomorphism \cite[\textsc{Theorem 8.2}]{weibel2013k}. 

Both $G_0$ and $K_0$ are functorial and admit pushforward along proper morphisms and pullback along flat morphisms \cite[\textsc{Lemma} 6.2.6, \S 2.8]{weibel2013k}. If $f:X\to Y$ is proper and $\mathcal E$ is a vector bundle on $X$, then the class $f_*[\mathcal E] \in K_0(Y)$ is defined as $\sum_{i\ge 0} (-1)^i[R^if_*\mathcal E]$ where $R^i$ is the $i^\text{th}$ higher direct image functor. In particular, if $i:Z\to X$ is the closed immersion of a subvariety, then $[i_*\cO_Z]$ is a class in $G_0(X)$. Since the inclusion $i$ is clear in most cases, we may denote this class simply by $[\cO_Z]$ in $G_0(X)$, called the \textit{fundamental class} of $Z$.

\begin{theorem}[Localization sequence, \cite{weibel2013k} \textsc{Application 6.4.2}] \label{thm:K0-localization} 
    Let $X$ be a quasiprojective variety over $\kk$, and let $i: Z\to X$ be inclusion of a closed subvariety $Z$ into $X$. Let $U = X-Z$ be the open complement and $j: U\to X$ be the inclusion. There is a right exact sequence of abelian groups $$G_0(Z)\overset{i_*}{\longrightarrow} G_0(X)\overset{j^*}{\longrightarrow} G_0(U) \to 0.$$
    If $X$ is in addition smooth, then there is a right exact sequence  \begin{equation}
            G_0(Z)\overset{i_*}{\longrightarrow} K_0(X)\overset{j^*}{\longrightarrow} K_0(U) \to 0,
        \end{equation}
where the flat pullback map $j^*$ is a homomorphism of rings. 
\end{theorem}

\begin{theorem}[Homotopy Invariance, \cite{weibel2013k} \textsc{Fundamental Theorem 6.5.1}]
 \label{thm:K0-homotopy}  Let $X$ be an $\kk$-scheme. If $E$ is a locally trivial bundle of affine spaces over $X$, then $G_0(E)\cong G_0(X)$. In particular, $G_0(\bA^n\times X) \cong G_0(X)$, and $K_0(\bA^n)\cong \ZZ$. 
\end{theorem}

\medskip\paragraph{\textsc{The Chow ring}} For a variety $X$, the Chow groups $CH_i(X)$ are the groups of $i$-dimensional algebraic cycles modulo rational equivalence; see, e.g., \cite[Chapter 1]{eisenbudharris3264}. When $X$ is smooth, we use the codimension grading $CH^p(X):=CH_{\dim X-p}(X)$, and the intersection product endows $CH^\bullet(X):=\bigoplus_p CH^p(X)$ with a graded-commutative ring structure (the Chow ring). The cycle class map \(
  cl: CH^\bullet(X)\to H^{2\bullet}(X(\CC);\ZZ)
\) is a functorial ring homomorphism sending an algebraic cycle to its singular cohomology class. 

\begin{small}
\begin{table}[ht]
    \centering
    \begin{tabular}{r|ccc} 
      & \textbf{$K$-theory} &  \textbf{Chow theory} & \textbf{Cohomology} \\ 
    \hline
    \textbf{Borel--Moore} & Coherent sheaves $G_0$ & Chow groups $CH_\bullet$ & Borel--Moore homology $H^{\textrm{BM}}_\bullet$\\ 
    \textbf{Cohomological} & Vector bundles $K_0$ & Chow ring $CH^\bullet$ & Singular cohomology $H^\bullet$
    \end{tabular}
    \medskip
    \caption{Comparison of topological invariants}
    \label{tab:cohomology-theories}
\end{table}
\end{small}

\paragraph{\textsc{Chern classes in the Grothendieck group $K_0$}}\label{subsec:chern-classes}
Let $X$ be a smooth algebraic variety of dimension $n$ and $\mathcal E$ a vector bundle of rank $r$. The \textit{$\lambda$-operations} on $K_0(X)$ are defined for all $i \geq 0$, where $\lambda_i: K_0(X)\to K_0(X)$ sends the class $[\mathcal E]$ to the class of the $i^\text{th}$ exterior power $  [\wedge^i \mathcal E]$. For all $0\leq i\leq r$, the $i^\text{th}$ \textit{$K$-theoretic Chern class} $c_i^K([\mathcal E])$ is the class $(-1)^i\left( \sum_{h = 1}^i \binom{i-1}{h-1}\lambda_h([\mathcal E]) - r \right)$ in $K_0(X)$, where $r = r\cdot [\cO_X]$ stands for the class $[\cO_X^{\oplus r}]$. In particular, the \textit{$K$-theoretic first Chern class} of a line bundle $L$ is  $c_1^K(\mathcal{L}) := 1 - [\mathcal{L}^\vee]$. The \textit{total Chern polynomial} is $c^K([\mathcal E])= \sum_{i = 0}^r c_i^K([\mathcal E])t^i$. We assume the following facts about $c_i^K$. 

\begin{itemize}
    \item (\textit{Whitney sum formula}) If $\mathcal E'\to \mathcal E\to \mathcal E''$ is a short exact sequence of vector bundles over $X$, then $c^K([\mathcal E]) = c^K([\mathcal E'])c^K([\mathcal E''])$.
    \item (\textit{Multiplicative group law for first Chern classes}) If $\mathcal L$ and $\mathcal M$ are line bundles, then it follows from the definition of $c_1^K$ that $$\begin{aligned}
        c_1^K(\mathcal L\otimes \mathcal M) & = 1-[\mathcal L^\vee]\cdot [\mathcal M^\vee] \\ & = 1 - (1 - c_1^K(\mathcal L))(1 - c_1^K(\mathcal M))\\ &= c_1^K(\mathcal L) + c_1^K(\mathcal M) - c_1^K(\mathcal L)\cdot c_1^K(\mathcal M).
    \end{aligned}$$ 
    \item (\textit{Normalization}) All Chern classes of the trivial bundle $\cO_X$ are zero.
\end{itemize}

\medskip\paragraph{\textsc{Chern classes in the Chow ring}}  Let $X$ be a smooth algebraic variety. For a vector bundle $\mathcal E$ of rank $n$, there exists a class $c_i(\mathcal E) \in C H^{i}(X)$ for each $0\leq i\leq n$, which represents the degeneracy locus where $n-i+1$ general sections of $\mathcal E$ become linearly dependent (\cite[\S 4]{fulton1984intersection}). The class $c_i(\mathcal E)$ is called the $i^\text{th}$ \textit{Chern class} of the bundle $\mathcal E$, and the generating function $c(\mathcal E):= \sum_{0\leq i\leq n} c_i(\mathcal E)t^i$ is called the \textit{total Chern class}. In this paper, we assume the following facts about Chern classes. Chow-theoretic Chern classes satisfy the same Whitney sum formula, normalization, and the additive group law as follows.

\begin{itemize}
    \item (\textit{Additive group law for first Chern classes}) If $\mathcal L$ and $\mathcal M$ are line bundles over $X$, then $c_1(\mathcal L\otimes \mathcal M) = c_1(\mathcal L) + c_1(\mathcal M)$.
\end{itemize}

\medskip\paragraph{\textsc{Chern classes in cohomology}}  Let $X(\CC)$ be a smooth complex algebraic variety. For a vector bundle $\mathcal E$ of rank $n$, there exists a cohomology class $c_i(\mathcal E) \in H^{2i}(X)$ for each $0\leq i\leq n$ satisfying the same set of axioms as Chern classes in the Chow ring. The cycle class map sends the Chow-theoretic Chern classes to the cohomological Chern classes.

\medskip\paragraph{\textsc{Cellular varieties}}\label{subsec:cellular-varieties}
A variety $X$ has a \textit{cell decomposition} if $X$ admits a filtration by closed subvarieties $$
X=X_n \supset X_{n-1} \supset \ldots \supset X_0 \supset X_{-1}=\varnothing
$$
where for all $i$, $X_i - X_{i-1}$ is a finite disjoint union of varieties $C_{ij}$, called the \textit{cells}, such that each cell is isomorphic to an affine space $\bA^{d_{ij}}$. 
Varieties with a cell decomposition are called \textit{cellular varieties.} Standard examples include projective spaces, Grassmannians, the flag variety, and many toric varieties. For more on cellular varieties, we refer our reader to \cite[Example 1.9.1]{fulton1984intersection} and \cite[\S 1]{eisenbudharris3264}.

\begin{theorem}
    \label{thm:cellular-finite-generation-of-cohomology} Let $X$ be a (possibly singular) cellular variety with cells $C_i$ indexed by some finite set $\mathcal{I}$. Let $E$ denote either the zeroth $G$-theory $G_0$, the Chow groups $CH_\bullet$, or singular homology $H_\bullet$. Then $E(X)$ is a free abelian group with basis $\{[\overline{C}_i]\mid i\in \mathcal{I}\}$ consisting of the closure of each cell. 
\end{theorem}

\begin{proof}
When $E$ (in the sense of Fulton's \textit{bivariant theories} \cite{fulton1984intersection}) is either $CH_\bullet$ or $H_\bullet$, the result is standard; see, for example, \cite[Example 19.1.11]{fulton1984intersection}. When $X$ is smooth, we may identify $E$ with its corresponding cohomology theory $K_0$, $CH^\bullet$, or $H^\bullet$. Karpenko \cite[Corollary 6.11]{Karpenko00cellular} proves that the geometric cohomology theory of a relative cellular space is a free module over the cohomology of the base. All three cohomology theories satisfy Karpenko's assumptions, and specializing to the case where the base is a point yields the desired statement. For the case when $E = G_0$ and $X$ is singular, see \Cref{thm:G0-free-abelian-when-cellular}.
\end{proof}

We are not aware of a reference that treats the $G_0$-version of \Cref{thm:cellular-finite-generation-of-cohomology} for singular cellular varieties specifically. For completeness, we include a proof in \Cref{subsec:key-lemmas-for-cellular} as \Cref{thm:G0-free-abelian-when-cellular}.

\begin{theorem}[Fulton \cite{fulton1984intersection} Example 19.1.11] \label{thm:cycle-class-map}
    Let $X$ be a cellular variety. The cycle class map $cl: CH_i(X) \to H_{2i}(X(\CC); \ZZ)$ is an isomorphism of graded modules. Furthermore, if $X$ is smooth, then $cl: CH^\bullet(X)\to H^\bullet(X(\CC); \ZZ)$ is a degree-doubling isomorphism of graded rings. 

\end{theorem}

\begin{example}[$K_0$ of projective spaces] \label{eg:K0(Pn)}
Since $\mathbb{P}^n$ has a cell decomposition $\PP^n = \coprod_{k = 0}^n \bA^k$ where $\overline{\bA^k} \cong \PP^k$, by \Cref{thm:cellular-finite-generation-of-cohomology} we have $G_0(\PP^n) \cong \ZZ\cdot \left\{[\cO_{\PP^0}], [\cO_{\PP^1}], \dots, [\cO_{\PP^n}]\right\}$. Since $\PP^n$ is smooth, there is an isomorphism $K_0(\PP^n) \cong G_0(\PP^n)$ as free abelian groups of rank $n+1$.

Next, we derive a presentation of $K_0(\PP^n)$ as a quotient ring with generator $\eta=1 - [\mathcal{O}_{\mathbb{P}^n}(1)]$. Let $S_n$ be the ring $\kk[x_0, x_1, \dots, x_n]$. We think of $\PP^n$ as $\operatorname{Proj} S_n$, and the cell closure $\PP^i$ naturally included as $\operatorname{Proj}S_i$ for $0\leq i\leq n$. The Koszul complex associated to the regular sequence $x_{i+1}, x_{i+2}, \dots, x_n$ in $S_n$ 
induces a resolution of $\cO_{\PP^i}$ 
$$
0 \rightarrow \cO_{\PP^n}(-n+i)^{\oplus \binom{n-i}{n-i}} \rightarrow \cdots \rightarrow \cO_{\PP^n}(-2)^{\oplus\binom{n-i}{2}} \rightarrow \cO_{\PP^n}(-1)^{\oplus\binom{n-i}{1}} \rightarrow \cO_{\PP^n} \rightarrow \cO_{\PP^i} \rightarrow 0
$$
given by the twisting sheaves $\cO_{\PP^n}(j)$. Using the fact that $[\cO_{\PP^n}(j)] = [\cO_{\PP^n}(1)]^j$, the Koszul resolution after dualizing induces the relation \begin{gather*}
 0 =    (-1)^{n-i}[\cO_{\PP^n}(1)]^{n-i} + \cdots - (n-i)[\cO_{\PP^n}(1)] + [\cO_{\PP^n}] + [\cO_{\PP^i}]  \\
    \iff [\cO_{\PP^i}] = (1-[\cO_{\PP^n}(1)])^{n-i} = \eta^{n-i}
\end{gather*}
in $K_0(\PP^n)$. Therefore, we call $\eta$ the \textit{hyperplane class}. In particular, the Koszul resolution of the structure sheaf $\cO_{\PP^n}$ gives the relation $\eta^{n+1} = 0$ in $K_0(\PP^n)$. Thus, the natural ring homomorphism $\varphi\colon \mathbb{Z}[x]\to K_0(\mathbb{P}^n)$ mapping $x$ to $ \eta = 1 - [\mathcal{O}_{\PP^n}(1)]$ factors through $\mathbb{Z}[x]/(x^{n+1})$. Since both are free abelian groups of rank $n+1$, we conclude that 
\begin{equation*}
    K_0(\PP^n) \cong \ZZ[\eta]/(\eta^{n+1}).
\end{equation*}
Furthermore, if we set $x_i:= c_1^K(\cL_i)$ where $\cL_i$ is the pullback of $\cO_{\PP^k}(1)$ along the $i^\text{th}$ coordinate projection on $\Pnk$, then the projective bundle formula in $K$-theory implies that \begin{equation*}
    K_0\left(\Pnk\right) \cong \frac{\ZZ[x_1, \dots, x_n]}{(x_1^k, \dots, x_n^k)},
\end{equation*}
and we denote the right-hand side by $S_{n,k}$. We will use this presentation in \Cref{sec:K0-of-Xnk} to describe $K_0(X_{n,k})$ explicitly.

\end{example}

\begin{example}[$K_0$ of the flag variety] \label{eg:K0(Fln)} Recall from \Cref{sec:intro} that there is an isomorphism \eqref{eqn:isos-of-Fln}
    \begin{equation} \label{eq: K0 Fln}
        K_0(Fl_n) \cong \frac{\ZZ[x_1,\dots, x_n]}{(e_1, e_2, \dots, e_n)},
    \end{equation}
    where $x_i =  c_1^K(\cL_i^\vee)$ is the $K$-theoretic first Chern class of the $i^\text{th}$ dual tautological line bundle over $Fl_n$. One way of proving the above isomorphism is to build $Fl_n$ as an iterated projective bundle and then use a rank argument based on the free-abelian nature of $K_0(Fl_n)$ implied by \Cref{thm:G0-free-abelian-when-cellular} \cite{Karoubi1978}. 
\end{example}

\subsection{Schubert Polynomials and Grothendieck Polynomials}

In this section, we state useful facts about Schubert and Grothendieck polynomials, degeneracy loci formulae, and the combinatorics of pipe dreams.

\label{subsec:schubert-polynomials}

Let $\mathcal{S}_n$ be the symmetric group on the letters $[n]$. A permutation $w\in \mathcal{S}_n$ is a bijection $w: [n]\to [n]$. The \textit{one-line notation} of $w$ is the ordered list $[w_1, w_2, \dots, w_n]$, where $w_i$ is the value $w(i)$. We may abbreviate the one-line notation as $w_1w_2\dots w_n$. Let $s_i$ denote the \textit{simple transposition} $(i\; i+1)$ which swaps the letters $i$ and $i+1$. Let $w_0\in \mathcal{S}_n$ be the permutation $[n, n-1, \ldots, 1]$. 

An \textit{inversion} of $w\in \mathcal{S}_n$ is an ordered tuple $(i,j)$ such that $1\leq i<j\leq n$ and $w_i > w_j$. Let $\operatorname{Inv}(w)$ be the set of inversions of $w$. The \textit{Lehmer code} of $w$ is the ordered list $c(w)$ where $c(w)_i = |\{j \in[n] \mid(i, j) \in \operatorname{Inv}(w)\}|$. The \textit{length} $\ell(w)$ of a permutation is the minimal number $r$ such that $w$ can be written as a product of $r$ simple transpositions. Equivalently, $\ell(w) = |\operatorname{Inv}(w)|$. 

The symmetric group $\mathcal{S}_n$ naturally acts on the polynomial ring $\ZZ[x_1, \dots, x_n]$ by permuting the $n$ variables. For any simple transposition $s_i$,  let the \textit{divided difference operator} $\partial_i$ be the operator on $\ZZ[x_1, \dots, x_n]$ defined by \begin{equation}\label{eq:divided-difference-operators}
    \partial_i f=\frac{f-s_i \cdot f}{x_i-x_{i+1}} .
\end{equation}

 \begin{definition}\label{def:schubert-poly}
    The \textit{Schubert polynomial} $\mathfrak{S}_w\in \ZZ[x_1, \dots, x_n]$ is defined recursively by Lascoux--Sch\"utzenberger \cite{lascoux1982polynomes} as
\begin{equation}
\begin{split}
    \mathfrak{S}_w(x_1, \dots, x_n)= \begin{cases}x_1^{n-1} x_2^{n-2} \cdots x_{n-1} & \text { if } w=w_0 \\ \partial_i \mathfrak{S}_{w s_i}(x_1, \dots, x_n) & \text { if } w(i)<w(i+1).\end{cases}
    \end{split}
\end{equation}
 The \textit{double Schubert polynomial} $\mathfrak{S}_w$ is defined recursively as 
    \begin{equation}
\begin{split}
  \mathfrak{S}_w\left(x_1, \ldots, x_n \mid  y_1, \ldots, y_n\right)= \begin{cases}\prod_{i+j \leq n}\left(x_i-y_j\right) & \text { if } w=w_0, \\ \partial_i \mathfrak{S}_{w s_i}(x_1, \dots, x_n \mid y_1, \dots, y_n) & \text { if } w(i)<w(i+1),\end{cases}
\end{split}
\end{equation}
where $\partial_i$ acts only on the $x$-variables.
\end{definition}

The \textit{rank table} of a permutation $w\in \mathcal{S}_n$ is the $n\times n$ array where $\operatorname{rk}(w)[i,j] = \# \{k \in [j]\mid w_k\leq i\}$. With respect to a fixed complete flag $E_\bullet \in Fl_n$, the \textit{Schubert cell} $C_w(E_\bullet)$ is the open subset of $Fl_n$ defined by the rank conditions $$C_w(E_\bullet):=\left\{F_{\bullet} \in Fl(n) \mid \operatorname{dim}\left(E_i \cap F_j\right)=\operatorname{rk}(w)[i, j] \text{ for all }1 \leq i, j \leq n\right\}.$$ Then, the \textit{Schubert variety} $X_w(E_\bullet)$ is the closure $\overline{C}_w(E_\bullet)$ inside the flag variety. Furthermore, the \textit{Schubert class} $[X_w(E_\bullet)]\in H^\bullet(Fl_n(\CC);\ZZ)$ does not depend on the choice of $E_\bullet$. 
\begin{theorem}[Lascoux--Sch\"utzenberger \cite{lascoux1982polynomes}]
    For $w\in \mathcal{S}_n$, the isomorphism $H^\bullet(Fl_n(\CC); \ZZ)\cong R_n$ sends the Schubert class $[X_w]$ to the Schubert polynomial $\mathfrak{S}_w(x_1, \dots, x_n)$.
\end{theorem}

One family of permutations that becomes important in the proofs of \Cref{thm:grothendieck-poly-of-word} and \Cref{thm:main} is $v^{(i)}:=12\dots \hat{i}\dots ni \in \mathcal{S}_n$, where $\widehat{i}$ means removing the letter $i$. For example, if $n=5$, then $v^{(3)} = 12453$. Their Schubert polynomials have a particularly nice form.

\begin{example}\label{eg:grassmannian-schubert}
    Let $n = 5$. The permutation $v^{(4)} = 12354$ has Schubert polynomial  $$\mathfrak{S}_{12354} =  x_1 + x_2 + x_3 + x_4 = e_1(x_1,\dots, x_4) $$
    in the cohomology ring $H^\bullet(Fl_5)$. In general, $v^{(i)}\in\mathcal{S}_n$ has Schubert polynomial \[\mathfrak{S}_{v^{(i)}} = e_{n-i}(x_1, \dots, x_{n-1})\in H^\bullet(Fl_n).\]
\end{example}
\label{subsec:grothendieck-polys}

\begin{definition}\label{def:grothendieck-poly}
Let the \textit{isobaric divided difference operator} be $\pi_i:=\partial_i\left(1-x_{i+1}\right)$, where $\partial_i$ is the divided difference operator in \Cref{def:schubert-poly}. For $w\in \mathcal{S}_n$, Lascoux--Sch{\"u}tzenberger \cite{lascoux1982polynomes} defined the \textit{Grothendieck polynomial}  recursively as 
\begin{equation}
\begin{split}
    \mathfrak{G}_w(x_1, \dots, x_n)= \begin{cases}x_1^{n-1} x_2^{n-2} \cdots x_{n-1} & \text { if } w=[n, n-1, \ldots, 1], \\ \pi_i \mathfrak{G}_{w s_i}(x_1, \dots, x_n) & \text { if } w(i)<w(i+1).\end{cases}
    \end{split}
\end{equation}
The \textit{double Grothendieck polynomial} is defined recursively as 
\begin{equation}
\begin{split}
   \mathfrak{G}_w\left(x_1, \ldots, x_n \mid  y_1, \ldots, y_n\right):= \begin{cases}\prod_{i+j \leq n}\left(x_i+y_j- x_i y_j\right) & \text { if } w=[n, n-1, \ldots, 1], \\ \pi_i \mathfrak{G}_{w s_i}(x_1, \dots, x_n \mid y_1,\dots,y_n) & \text { if } w(i)<w(i+1),\end{cases}
    \end{split}
\end{equation}
where $\pi_i$ again acts only on the $x$-variables.
  
\end{definition}

\begin{theorem}[Lascoux--Sch\"utzenberger \cite{LascouxSchutzenberger1982Grothendieck}]
    For $w\in \mathcal{S}_n$, the isomorphism $K_0(Fl_n)\cong R_n$ sends the $K$-theoretic Schubert class $[\cO_{X_w}]$ to the Grothendieck polynomial $\mathfrak{G}_w(x_1, \dots, x_n)$.
\end{theorem}

\begin{example}\label{eg:grassmannian-grothendieck}
   The permutation $v^{(4)} = 12354\in\mathcal{S}_5$ has Grothendieck polynomial  $$\begin{aligned}
        \mathfrak{G}_{12354} &=  x_1 + x_2 + x_3 + x_4 - x_1 x_2  - x_1 x_3 - x_2 x_3 - x_3 x_4 - x_1 x_4 - x_2 x_4 \\ & + x_1x_2x_3+ x_1 x_2 x_4  + x_1 x_3 x_4 + x_2 x_3 x_4 - x_1 x_2 x_3 x_4  \\ & = e_1(x_1,\dots, x_4) - e_2(x_1,\dots, x_4) + e_3(x_1,\dots, x_4) - e_4(x_1,\dots, x_4). 
    \end{aligned}$$
Comparing with \Cref{eg:grassmannian-schubert}, we see that the lowest degree component of $\mathfrak{G}_{12354}$ is $\mathfrak{S}_{12354}$, and the higher degree terms form an alternating sum of the other elementary symmetric polynomials. In general, the Grothendieck polynomial of the permutation $v^{(i)}$ can be computed combinatorially as in \Cref{lem:G-of-vertical-strip}. For any permutation $w$, the lowest degree component of $\mathfrak{G}_w$ always agrees with the Schubert polynomial $\mathfrak{S}_w$.
\end{example}

\begin{lemma}
  \label{lem:G-of-vertical-strip}
    The Grothendieck polynomial for the permutation $v^{(i)} =[12 \ldots \hat{i} \ldots n i] \in \mathcal{S}_{n}$ is the alternating sum of elementary symmetric polynomials
    $$\begin{aligned}
        \mathfrak{G}_{v^{(i)}} & = e_{n+1-i}-g_{n+2-i}e_{n+2-i} + \cdots + (-1)^{i+1}g_{n}e_n\\
        & = e_{n+1-i} + \sum_{n+1-i<j\leq n} (-1)^{i+j-n-1}g_je_j,
    \end{aligned}$$
    in the variables $x_1, \dots, x_{n-1}$, for some non-negative coefficients $g_j$. 
\end{lemma}

\begin{proof}
Let $w$ be a Grassmannian permutation, which means $w$ contains at most one descent. It is well-known that Grassmannian permutations with descent position $i$ can be canonically identified with partitions fitting into the $i\times (n-i)$ box, which forms a poset called Young's lattice, where the ordering is given by containment of Ferres diagrams. For a partition $\lambda\subseteq [i]\times [n-i]$, let $\widehat{\lambda}$ be the unique maximal partition in Young's lattice with $i$ rows, obtainable from $\lambda$ by adding at most $j-1$ boxes to its $j^\text{th}$ row for $2\leq j\leq i$. Lenart \cite[Theorem 2.2]{Lenart2000} proved that the Grothendieck polynomial $\mathfrak{G}_w$ expands as a sum of Schur polynomials 
$$\mathfrak{G}_w=\sum_{\lambda \subseteq \mu \subseteq \hat{\lambda}}(-1)^{|\mu|-|\lambda|} g_{\lambda \mu} s_\mu ,$$
where the coefficients $g_{\lambda\mu}$ are non-negative integers corresponding to certain combinatorial counts.

The diagram of the permutation $v^{(i)}$ corresponds to the partition $\lambda^{(i)} = (1^{n+1-i})$. Then, $\widehat{\lambda^{(i)}}$ is the vertical strip $(1^n)$, and the interval $[\lambda^{(i)}, \widehat{\lambda^{(i)}}]$ consists of all vertical strips $\mu^{(j)} = (1^{j})$ for $n+1-i \leq j \leq n$. Therefore, $\mathfrak{G}_{v^{(i)}}$ is an alternating sum of elementary symmetric polynomials $e_{n-i+1}, \dots, e_n$, in the form $  \mathfrak{G}_{v^{(i)}} =  \sum_{n+1-i\leq j\leq n} (-1)^{i+j-n-1}g_je_j$. In particular, it follows from the definition of $g_{\lambda\mu}$ that the leading coefficient $g_{n-i+1}$ is $1$. 
\end{proof}

\vspace{1em}
\medskip\paragraph{\textsc{Degeneracy Loci Formulae}} A \textit{flagged vector bundle} is a vector bundle $\mathcal{E}_\bullet\to X$ with the continuous choice of a complete flag in each fiber. A map of flagged vector bundles $f: \mathcal E_\bullet \to \mathcal  F_\bullet$ is the data of an increasing chain of inclusions and a descending chain of quotients
\[
\quad 0 = \mathcal  E_0 \hookrightarrow\mathcal  E_1 \hookrightarrow \cdots \hookrightarrow \mathcal E_n \overset{f}{\longrightarrow} \mathcal F_k \twoheadrightarrow \mathcal F_{k-1} \twoheadrightarrow \cdots \twoheadrightarrow \mathcal F_0 = 0,
\]
where each $\mathcal E_i$ is a subbundle of rank $i$ and $\mathcal F_j$ a subbundle of rank $j$. For each pair $(i,j)$, let $f^{ij}$ be the restriction of $f$ from $\mathcal E_i$ to $\mathcal F_j$. For a point $x$ in the base space $X$, the map $f_x^{ij}$ on the fibers is a linear map of vector spaces $(\mathcal E_i)_x\to (\mathcal E_j)_x$. The rank table of $f$ at $x$ is the $n\times k$ array with entries $\operatorname{rk}_{f_x}[i,j] = \operatorname{rank} f_x^{ij}$.

Now fix a permutation $w \in \mathcal S_n$. Recall the rank table of $w$ is the $n\times n$ array with entries
\[
\operatorname{rk}(w)[i,j] := \#\{\ell \in \{1,\dots,j\} \mid w(\ell) \le i\}.
\]
Given a morphism of flagged bundles $f:\mathcal E_\bullet \to \mathcal F_\bullet$, the \emph{flagged degeneracy locus} of $f$ with respect to $w$ is the closed subset
\[
\Omega_w(f)
= \bigl\{\, x\in X \;\big|\; \operatorname{rank} f_x^{ij} \le \operatorname{rk}(w)[i,j] \text{ for all } 1\le i,j\le n \,\bigr\}.
\]

Associated to the flags $\mathcal E_\bullet$ and $\mathcal F_\bullet$ are natural line bundles.  For each $i$ the quotient $\mathcal E_i/\mathcal E_{i-1}$ is a line bundle on $X$, and for each $j$ the kernel of the projection $\mathcal F_j \to \mathcal F_{j-1}$ is also a line bundle.  We write
\[
b_j := c_1(\mathcal E_j/\mathcal E_{j-1}), \qquad
a_i := c_1(\ker(\mathcal F_i \to \mathcal F_{i-1}))
\]
for their first Chern classes in the Chow ring, and
\[
b_j^K := c_1^K(\mathcal E_j/\mathcal E_{j-1}), \qquad
a_i^K := c_1^K(\ker(\mathcal F_i \to \mathcal F_{i-1}))
\]
for the corresponding $K$-theoretic first Chern classes in $K_0(X)$. The classical degeneracy loci formula of Fulton identifies the Chow class of $\Omega_w(f)$ with a double Schubert polynomial.

\begin{theorem}[Fulton {\cite[Thm.~8.2]{Fulton1991degeneracy}}]
\label{thm:Schubert-degeneracy-loci-formula}
If $\Omega_w(f)$ has codimension equal to the length $\ell(w)$ of $w$, then its class in $CH^\bullet(X)$ is represented by the double Schubert polynomial:
\[
[\Omega_w(f)] = \mathfrak{S}_w(a_1,\dots,a_n \mid b_1,\dots,b_n).
\]
\end{theorem}

Fulton and Lascoux later proved an analogue in $K$-theory, expressing the structure sheaf class of $\Omega_w(f)$ in terms of a double Grothendieck polynomial.

\begin{theorem}[Fulton--Lascoux {\cite[Thm.~3]{FultonLascoux1994FlagBundle}}]
\label{thm:Grothendieck-degeneracy-loci-formula}
If $\Omega_w(f)$ has the expected codimension, then in $K_0(X)$ its structure sheaf class is given by
\[
[\cO_{\Omega_w(f)}]
= \mathfrak{G}_w(a_1^K,\dots,a_n^K \mid b_1^K,\dots,b_n^K).
\]
\end{theorem}
These two formulas are the main inputs in our identification of classes of Pawlowski--Rhoades varieties $X_w$ with Schubert and Grothendieck polynomials associated to words.

\subsection{Classical and Bumpless pipe dreams}
\label{subsec:classical-PD} Pipe dreams are combinatorial objects closely related to Schubert and Grothendieck polynomials. We begin this section by recollecting some basic facts about classical pipe dreams. A \emph{(classical) pipe dream} $P$ is a finite subset of $\ZZ_{+}\times\ZZ_{+}$. One depicts $P$ typically by arranging $\ZZ_{+}\times\ZZ_{+}$ as a two-dimensional grid, placing a \crosstile-tile at each $(i,j)\in P$ and an \bt-tile at each $(i,j)\notin P$. The \crosstile-tile is called a \textit{cross} and the \bt-tile is called a \textit{bump}. The contents of tiles are connected into \textit{pipes} that run from the left boundary to the top boundary. The pipes are numbered according to their entering positions on the left boundary.

If one labels the top boundary by $1,2,3,\dots$ and reads along the left boundary from top to bottom the labels of the pipes that exit there, then one obtains a permutation $w$ of the positive integers that fixes all but finitely many values. We call $w$ the permutation of $P$. The pipe dream $P$ is \emph{reduced} if $\ell(w)=|P|$, or equivalently, if no two pipes cross twice. Note that when tracing a pipe in a non-reduced pipe dream, all but the first cross-tiles are interpreted as bump-tiles and hence ignored. Write $\operatorname{PD}(w)$ for the set of reduced pipe dreams for $w$ and $\overline{\operatorname{PD}}(w)$ for all pipe dreams for $w$. Since only finitely many crosses can appear, a finite triangular portion of the grid suffices to draw any diagram. The \textit{top pipe dream} $P_{\textrm{top}}$ is obtained by placing $c(w^{-1})_i$-many top-justified \crosstile-tiles in column $i$, where $c(w^{-1})$ is the Lehmer code. For more on classical pipe dreams, see \cite[\S 4]{textchapter}. 

Classical pipe dreams were first introduced by Bergeron-Billey \cite{BergeronBilley1993RC} under the name ``RC-graphs'' which encode the information of Stanley's \textit{compatible pairs} \cite{StanleyCompatiblePairs} in a visual way. Knutson-Miller \cite{KnutsonMiller2005} introduced the term ``pipe dreams.'' They considered non-reduced pipe dreams and discussed the algebraic and geometric information they contain.

\begin{example}
     \label{eg:PD-for-24153}
    Let $w\in S_5$ be the permutation $24153$. All reduced classical pipe dreams of $w$ are as follows. The top pipe dream $P_{\textrm{top}}$ is on the top left.
    
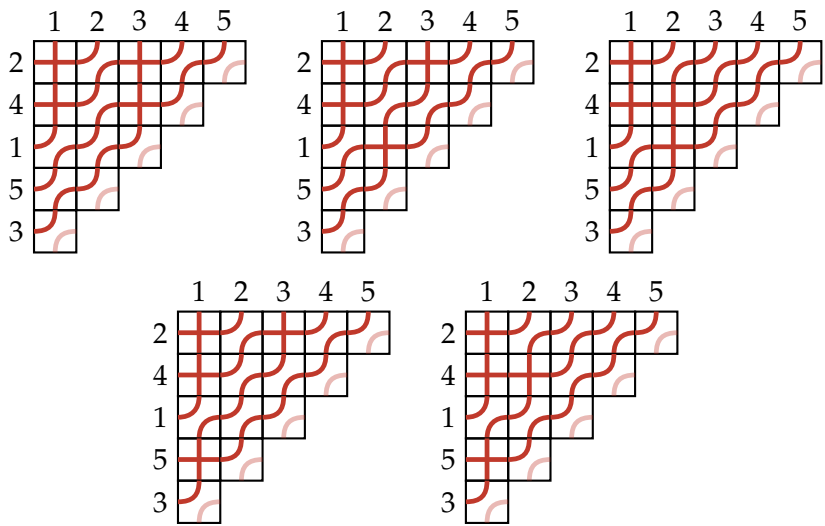
\begin{scriptsize}
        \[\begin{gathered}
        \tikzpicture
  \PipeStart{5}
  \PipePlace{1}{1}{\pipeCross}
  \PipePlace{1}{2}{\pipeBump}
  \PipePlace{1}{3}{\pipeCross}
  \PipePlace{1}{4}{\pipeBump}
  \PipePlace{1}{5}{\pipeElbow}
  \PipePlace{2}{1}{\pipeCross}
  \PipePlace{2}{2}{\pipeBump}
  \PipePlace{2}{3}{\pipeCross}
  \PipePlace{2}{4}{\pipeElbow}
  \PipePlace{3}{1}{\pipeBump}
  \PipePlace{3}{2}{\pipeBump}
  \PipePlace{3}{3}{\pipeElbow}
  \PipePlace{4}{1}{\pipeBump}
  \PipePlace{4}{2}{\pipeElbow}
  \PipePlace{5}{1}{\pipeElbow}

  \PipeLabels{1,2,3,4,5}{2,4,1,5,3}
  \PipeEnd
\endtikzpicture 
\qquad
\tikzpicture
  \PipeStart{5}
  \PipePlace{1}{1}{\pipeCross}
  \PipePlace{1}{2}{\pipeBump}
  \PipePlace{1}{3}{\pipeCross}
  \PipePlace{1}{4}{\pipeBump}
  \PipePlace{1}{5}{\pipeElbow}
  \PipePlace{2}{1}{\pipeCross}
  \PipePlace{2}{2}{\pipeBump}
  \PipePlace{2}{3}{\pipeBump}
  \PipePlace{2}{4}{\pipeElbow}
  \PipePlace{3}{1}{\pipeBump}
  \PipePlace{3}{2}{\pipeCross}
  \PipePlace{3}{3}{\pipeElbow}
  \PipePlace{4}{1}{\pipeBump}
  \PipePlace{4}{2}{\pipeElbow}
  \PipePlace{5}{1}{\pipeElbow}

  \PipeLabels{1,2,3,4,5}{2,4,1,5,3}
  \PipeEnd
\endtikzpicture
\qquad
\tikzpicture
  \PipeStart{5}
  \PipePlace{1}{1}{\pipeCross}
  \PipePlace{1}{2}{\pipeBump}
  \PipePlace{1}{3}{\pipeBump}
  \PipePlace{1}{4}{\pipeBump}
  \PipePlace{1}{5}{\pipeElbow}
  \PipePlace{2}{1}{\pipeCross}
  \PipePlace{2}{2}{\pipeCross}
  \PipePlace{2}{3}{\pipeBump}
  \PipePlace{2}{4}{\pipeElbow}
  \PipePlace{3}{1}{\pipeBump}
  \PipePlace{3}{2}{\pipeCross}
  \PipePlace{3}{3}{\pipeElbow}
  \PipePlace{4}{1}{\pipeBump}
  \PipePlace{4}{2}{\pipeElbow}
  \PipePlace{5}{1}{\pipeElbow}

  \PipeLabels{1,2,3,4,5}{2,4,1,5,3}
  \PipeEnd
\endtikzpicture \\
\tikzpicture
  \PipeStart{5}
  \PipePlace{1}{1}{\pipeCross}
  \PipePlace{1}{2}{\pipeBump}
  \PipePlace{1}{3}{\pipeCross}
  \PipePlace{1}{4}{\pipeBump}
  \PipePlace{1}{5}{\pipeElbow}
  \PipePlace{2}{1}{\pipeCross}
  \PipePlace{2}{2}{\pipeBump}
  \PipePlace{2}{3}{\pipeBump}
  \PipePlace{2}{4}{\pipeElbow}
  \PipePlace{3}{1}{\pipeBump}
  \PipePlace{3}{2}{\pipeBump}
  \PipePlace{3}{3}{\pipeElbow}
  \PipePlace{4}{1}{\pipeCross}
  \PipePlace{4}{2}{\pipeElbow}
  \PipePlace{5}{1}{\pipeElbow}

  \PipeLabels{1,2,3,4,5}{2,4,1,5,3}
  \PipeEnd
\endtikzpicture
\qquad
\tikzpicture
  \PipeStart{5}
  \PipePlace{1}{1}{\pipeCross}
  \PipePlace{1}{2}{\pipeBump}
  \PipePlace{1}{3}{\pipeBump}
  \PipePlace{1}{4}{\pipeBump}
  \PipePlace{1}{5}{\pipeElbow}
  \PipePlace{2}{1}{\pipeCross}
  \PipePlace{2}{2}{\pipeCross}
  \PipePlace{2}{3}{\pipeBump}
  \PipePlace{2}{4}{\pipeElbow}
  \PipePlace{3}{1}{\pipeBump}
  \PipePlace{3}{2}{\pipeBump}
  \PipePlace{3}{3}{\pipeElbow}
  \PipePlace{4}{1}{\pipeCross}
  \PipePlace{4}{2}{\pipeElbow}
  \PipePlace{5}{1}{\pipeElbow}

  \PipeLabels{1,2,3,4,5}{2,4,1,5,3}
  \PipeEnd
\endtikzpicture
    \end{gathered}
\]
    \end{scriptsize}

\end{example}

The \emph{monomial weight} of a pipe dream $P$ is the product of the row-indexed variables over the its crossings,
\[
x^{P}:=\prod_{(i,j)\in P}x_{i}.
\]
With a second set of variables $y_j$, one can also record the column indices in the product
\[
(x-y)^{P}:=\prod_{(i,j)\in P}(x_{i}-y_{j})
\]
and define a `$K$-theoretic weight'
\[(x/y)^P := \prod_{(i,j\in P)}(x_1 + y_j - x_1y_j).\]

\begin{theorem}[Bergeron-Billey \cite{BergeronBilley1993RC} Theorem 3.7; Fomin--Kirillov \cite{FominKirillovDoubleSchubert} Proposition 6.2; Fomin--Kirillov \cite{FominKirillov1994YangBaxterGrothendieck} 2.3 Theorem; \cite{Buch2005DoubleGrothendieck} Corollary 2.2]\label{thm:pipe dreams} 
For $w\in S_n$, its Schubert polynomial is equal to
\[
\mathfrak{S}_w(x_1,\dots,x_n)=\sum_{P\in\operatorname{PD}(w)} x^{P}.
\]
The double Schubert polynomial of $w$ is equal to \[\mathfrak{S}_w(x\mid y) = \sum_{P\in \operatorname{PD}(w)}(x-y)^P.\]
The Grothendieck polynomial of $w$ is equal to \[\mathfrak{G}_w(x_1, \dots, x_n) = \sum_{P \in \overline{PD}(w)} (-1)^{\# P - \ell(w)}x^P.\]
The double Grothendieck polynomial of $w$ is equal to \[\mathfrak{G}_w(x\mid y) = \sum_{P \in \overline{PD}(w)} (-1)^{\# P - \ell(w)}(x/y)^P.\]
\end{theorem}
In other words, Schubert and Grothendieck polynomials are monomial-weight generating functions for classical pipe dreams. For instance, the five reduced pipe dreams in \Cref{eg:PD-for-24153} have monomial weights \[x_1^2x_2^2, x_1^2x_2x_3, x_1^2x_2x_4, x_1x_2^2x_3, x_1x_2^2x_4,\]
and the Schubert polynomial for $w = 24153$ is equal to the sum of these five monomial terms, \[\mathfrak{S}_w(x) = x_1^2x_2^2+ x_1^2x_2x_3+ x_1^2x_2x_4+ x_1x_2^2x_3+x_1x_2^2x_4.\]

\textit{Chute moves} are permutation-preserving local moves on $\overline{\operatorname{PD}}(w)$ introduced in \cite{BergeronBilley1993RC}. For $P \in \overline{\operatorname{PD}}(w)$, suppose there are positive integers $i,j,k$ such that $i\leq j$, the rectangle $[k,k+1]\times [i+1,j] \subseteq P$, the NE corner $(k,j+1)$ is in $P$ but all other corners $(k,i),(k+1,i), (k+1, j+1)\notin P$. Then, the local move sending $P$ to $(P-(k,j+1)) \cup( k+1,i)$ is a \textit{chute move}. The move sending $P$ to $P \cup(k+1,i)$ is a \textit{K-theoretic chute move}. See \Cref{fig:chute-move} for an illustration. It is a theorem of Bergeron-Billey \cite[Theorem 3.7]{BergeronBilley1993RC} that all reduced pipe dreams in $\operatorname{PD}(w)$ can be obtained by performing chute moves on $P_{\textrm{top}}$, and all pipe dreams in $\overline{\operatorname{PD}}(w)$ can be obtained by ordinary and $K$-theoretic chute moves on $P_{\textrm{top}}$.

\begin{figure}[h]
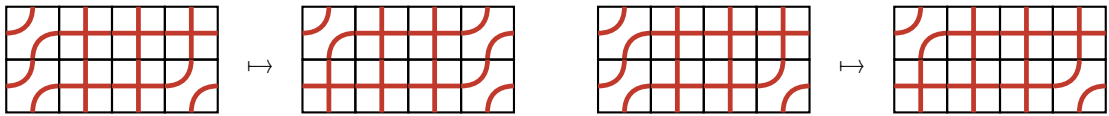

    \centering
    \begin{footnotesize}
    \[\tikzpicture
  \PipeStart{6}
  \PipePlace{1}{1}{\pipeBump}
  \PipePlace{1}{2}{\pipeCross}
  \PipePlace{1}{3}{\pipeCross}
  \PipePlace{1}{4}{\pipeCross}
  \PipePlace{2}{1}{\pipeBump}
  \PipePlace{2}{2}{\pipeCross}
  \PipePlace{2}{3}{\pipeCross}
  \PipePlace{2}{4}{\pipeBump}
  \PipeEnd
\endtikzpicture \quad \raisebox{1.6em}{$\mapsto$} \quad
\tikzpicture
  \PipeStart{6}
  \PipePlace{1}{1}{\pipeBump}
  \PipePlace{1}{2}{\pipeCross}
  \PipePlace{1}{3}{\pipeCross}
  \PipePlace{1}{4}{\pipeBump}
  \PipePlace{2}{1}{\pipeCross}
  \PipePlace{2}{2}{\pipeCross}
  \PipePlace{2}{3}{\pipeCross}
  \PipePlace{2}{4}{\pipeBump}
  \PipeEnd
\endtikzpicture
\quad \qquad
\tikzpicture
  \PipeStart{6}
  \PipePlace{1}{1}{\pipeBump}
  \PipePlace{1}{2}{\pipeCross}
  \PipePlace{1}{3}{\pipeCross}
  \PipePlace{1}{4}{\pipeCross}
  \PipePlace{2}{1}{\pipeBump}
  \PipePlace{2}{2}{\pipeCross}
  \PipePlace{2}{3}{\pipeCross}
  \PipePlace{2}{4}{\pipeBump}
  \PipeEnd
\endtikzpicture \quad \raisebox{1.6em}{$\mapsto$} \quad
\tikzpicture
  \PipeStart{6}
  \PipePlace{1}{1}{\pipeBump}
  \PipePlace{1}{2}{\pipeCross}
  \PipePlace{1}{3}{\pipeCross}
  \PipePlace{1}{4}{\pipeCross}
  \PipePlace{2}{1}{\pipeCross}
  \PipePlace{2}{2}{\pipeCross}
  \PipePlace{2}{3}{\pipeCross}
  \PipePlace{2}{4}{\pipeBump}
  \PipeEnd
\endtikzpicture\]
\end{footnotesize}
    \caption{Visualization of chute moves (left) and $K$-theoretic chute moves (right).}
    \label{fig:chute-move}
\end{figure}

Next, we recollect some facts about bumpless pipe dreams. A \emph{bumpless pipe dream (BPD)} $B$ is a tiling of the grid $\ZZ_{+}\times\ZZ_{+}$ in matrix coordinates by the following six tiles
\begin{large}
\[\bpdNW \quad \bpdSE \quad \bpdblank \quad \bpdplus \quad \bpdh \quad \bpdv  \]
\end{large}
that route pipes monotonically in the northeast direction. Similar to classical pipe dreams, the pipes of a BPD are numbered according to their entering positions on the right boundary. A bumpless pipe dream is \emph{reduced} if no two pipes cross twice. It encodes a wiring diagram of a permutation of the positive integers with all but finitely many non-fixed values by requiring that pipe $i$ enters along the left boundary at height $i$ and exits along the top boundary at column $w(i)$. Redundant crossings in a non-reduced BPD are similarly ignored for the purpose of determining the permutation. The term “bumpless” reflects that the bump tile~\bt \ from classical pipe dreams is forbidden. For $w\in \mathcal{S}_n$, the \textit{diagram BPD}, denoted $B_{\textrm{diagram}}$, is the unique BPD having SE-elbow tiles at exactly $(w_i,i)$ for all $i$. For a permutation $w$, let $\operatorname{BPD}(w)$ be the set of reduced bumpless pipe dreams of $w$ and $\overline{\operatorname{BPD}}(w)$ the set of all BPDs for $w$.  For $w\in S_n$, one typically draws a diagram inside an $n\times n$ grid, where all possible crosses for $\operatorname{BPD}(w)$ take place. See \Cref{eg:BPD-for-24153} for an illustration.

Lam-Lee-Shimozono introduced bumpless pipe dreams in their study of the infinite flag variety and back-stable Schubert calculus \cite{LLS}. Subsequent work has developed combinatorial moves, algebraic structures, and geometric interpretations for these tilings, and has clarified their connections and contrasts with classical pipe dreams, including distinct factorizations for double Schubert polynomials \cite{weigandt2021bumpless, KleinWeigandtGroebner, HuangMonk, GaoHuangCanonical, LeOuyangRestivoTaoZhangQuantum, KnutsonHybrid}.

\begin{example}\label{eg:BPD-for-24153}
Below are all bumpless pipe dreams of $w = 24153$. The diagram BPD $B_{\textrm{diagram}}$ is at the top left. Note that the one on the bottom right is non-reduced.
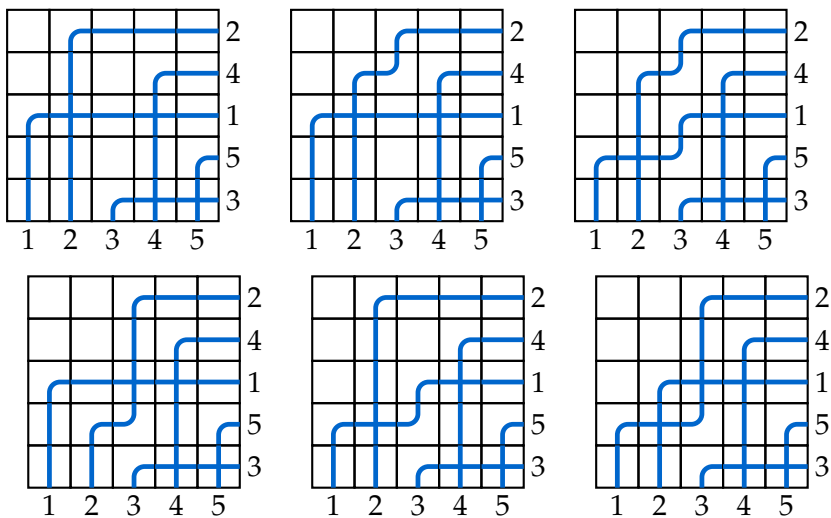
\begin{scriptsize}
    \[\begin{gathered} \tikzpicture
  \BPDStart{5}{5}
  \BPDPlace{1}{1}{\blank}
  \BPDPlace{1}{2}{\SE}
  \BPDPlace{1}{3}{\hor}
  \BPDPlace{1}{4}{\hor}
  \BPDPlace{1}{5}{\hor}
  \BPDPlace{2}{1}{\blank}
  \BPDPlace{2}{2}{\ver}
  \BPDPlace{2}{3}{\blank}
  \BPDPlace{2}{4}{\SE}
  \BPDPlace{2}{5}{\hor}
  \BPDPlace{3}{1}{\SE}
  \BPDPlace{3}{2}{\+}
  \BPDPlace{3}{3}{\hor}
  \BPDPlace{3}{4}{\+}
  \BPDPlace{3}{5}{\hor}
  \BPDPlace{4}{1}{\ver}
  \BPDPlace{4}{2}{\ver}
  \BPDPlace{4}{3}{\blank}
  \BPDPlace{4}{4}{\ver}
  \BPDPlace{4}{5}{\SE}
  \BPDPlace{5}{1}{\ver}
  \BPDPlace{5}{2}{\ver}
  \BPDPlace{5}{3}{\SE}
  \BPDPlace{5}{4}{\+}
  \BPDPlace{5}{5}{\+}
  \BPDEnd

  \begin{scope}[x=2em,y=2em, yscale=-1, shift={(0,5)}]
    \node[below] at (0.5,5) {\small 1};
    \node[below] at (1.5,5) {\small 2};
    \node[below] at (2.5,5) {\small 3};
    \node[below] at (3.5,5) {\small 4};
    \node[below] at (4.5,5) {\small 5};
  \end{scope}
  \begin{scope}[x=2em,y=2em, yscale=-1, shift={(0,5)}]
    \node[right] at (5,0.5) {\small 2};
    \node[right] at (5,1.5) {\small 4};
    \node[right] at (5,2.5) {\small 1};
    \node[right] at (5,3.5) {\small 5};
    \node[right] at (5,4.5) {\small 3};
  \end{scope}
\endtikzpicture  \qquad
\tikzpicture
  \BPDStart{5}{5}
  \BPDPlace{1}{1}{\blank}
  \BPDPlace{1}{2}{\blank}
  \BPDPlace{1}{3}{\SE}
  \BPDPlace{1}{4}{\hor}
  \BPDPlace{1}{5}{\hor}
  \BPDPlace{2}{1}{\blank}
  \BPDPlace{2}{2}{\SE}
  \BPDPlace{2}{3}{\NW}
  \BPDPlace{2}{4}{\SE}
  \BPDPlace{2}{5}{\hor}
  \BPDPlace{3}{1}{\SE}
  \BPDPlace{3}{2}{\+}
  \BPDPlace{3}{3}{\hor}
  \BPDPlace{3}{4}{\+}
  \BPDPlace{3}{5}{\hor}
  \BPDPlace{4}{1}{\ver}
  \BPDPlace{4}{2}{\ver}
  \BPDPlace{4}{3}{\blank}
  \BPDPlace{4}{4}{\ver}
  \BPDPlace{4}{5}{\SE}
  \BPDPlace{5}{1}{\ver}
  \BPDPlace{5}{2}{\ver}
  \BPDPlace{5}{3}{\SE}
  \BPDPlace{5}{4}{\+}
  \BPDPlace{5}{5}{\+}
  \BPDEnd

  \begin{scope}[x=2em,y=2em, yscale=-1, shift={(0,5)}]
    \node[below] at (0.5,5) {\small 1};
    \node[below] at (1.5,5) {\small 2};
    \node[below] at (2.5,5) {\small 3};
    \node[below] at (3.5,5) {\small 4};
    \node[below] at (4.5,5) {\small 5};
  \end{scope}
  \begin{scope}[x=2em,y=2em, yscale=-1, shift={(0,5)}]
    \node[right] at (5,0.5) {\small 2};
    \node[right] at (5,1.5) {\small 4};
    \node[right] at (5,2.5) {\small 1};
    \node[right] at (5,3.5) {\small 5};
    \node[right] at (5,4.5) {\small 3};
  \end{scope}
\endtikzpicture \qquad 
\tikzpicture
  \BPDStart{5}{5}
  \BPDPlace{1}{1}{\blank}
  \BPDPlace{1}{2}{\blank}
  \BPDPlace{1}{3}{\SE}
  \BPDPlace{1}{4}{\hor}
  \BPDPlace{1}{5}{\hor}
  \BPDPlace{2}{1}{\blank}
  \BPDPlace{2}{2}{\SE}
  \BPDPlace{2}{3}{\NW}
  \BPDPlace{2}{4}{\SE}
  \BPDPlace{2}{5}{\hor}
  \BPDPlace{3}{1}{\blank}
  \BPDPlace{3}{2}{\ver}
  \BPDPlace{3}{3}{\SE}
  \BPDPlace{3}{4}{\+}
  \BPDPlace{3}{5}{\hor}
  \BPDPlace{4}{1}{\SE}
  \BPDPlace{4}{2}{\+}
  \BPDPlace{4}{3}{\NW}
  \BPDPlace{4}{4}{\ver}
  \BPDPlace{4}{5}{\SE}
  \BPDPlace{5}{1}{\ver}
  \BPDPlace{5}{2}{\ver}
  \BPDPlace{5}{3}{\SE}
  \BPDPlace{5}{4}{\+}
  \BPDPlace{5}{5}{\+}
  \BPDEnd

  \begin{scope}[x=2em,y=2em, yscale=-1, shift={(0,5)}]
    \node[below] at (0.5,5) {\small 1};
    \node[below] at (1.5,5) {\small 2};
    \node[below] at (2.5,5) {\small 3};
    \node[below] at (3.5,5) {\small 4};
    \node[below] at (4.5,5) {\small 5};
  \end{scope}
  \begin{scope}[x=2em,y=2em, yscale=-1, shift={(0,5)}]
    \node[right] at (5,0.5) {\small 2};
    \node[right] at (5,1.5) {\small 4};
    \node[right] at (5,2.5) {\small 1};
    \node[right] at (5,3.5) {\small 5};
    \node[right] at (5,4.5) {\small 3};
  \end{scope}
\endtikzpicture \qquad \\
\tikzpicture
  \BPDStart{5}{5}
  \BPDPlace{1}{1}{\blank}
  \BPDPlace{1}{2}{\blank}
  \BPDPlace{1}{3}{\SE}
  \BPDPlace{1}{4}{\hor}
  \BPDPlace{1}{5}{\hor}
  \BPDPlace{2}{1}{\blank}
  \BPDPlace{2}{2}{\blank}
  \BPDPlace{2}{3}{\ver}
  \BPDPlace{2}{4}{\SE}
  \BPDPlace{2}{5}{\hor}
  \BPDPlace{3}{1}{\SE}
  \BPDPlace{3}{2}{\hor}
  \BPDPlace{3}{3}{\+}
  \BPDPlace{3}{4}{\+}
  \BPDPlace{3}{5}{\hor}
  \BPDPlace{4}{1}{\ver}
  \BPDPlace{4}{2}{\SE}
  \BPDPlace{4}{3}{\NW}
  \BPDPlace{4}{4}{\ver}
  \BPDPlace{4}{5}{\SE}
  \BPDPlace{5}{1}{\ver}
  \BPDPlace{5}{2}{\ver}
  \BPDPlace{5}{3}{\SE}
  \BPDPlace{5}{4}{\+}
  \BPDPlace{5}{5}{\+}
  \BPDEnd

  \begin{scope}[x=2em,y=2em, yscale=-1, shift={(0,5)}]
    \node[below] at (0.5,5) {\small 1};
    \node[below] at (1.5,5) {\small 2};
    \node[below] at (2.5,5) {\small 3};
    \node[below] at (3.5,5) {\small 4};
    \node[below] at (4.5,5) {\small 5};
  \end{scope}
  \begin{scope}[x=2em,y=2em, yscale=-1, shift={(0,5)}]
    \node[right] at (5,0.5) {\small 2};
    \node[right] at (5,1.5) {\small 4};
    \node[right] at (5,2.5) {\small 1};
    \node[right] at (5,3.5) {\small 5};
    \node[right] at (5,4.5) {\small 3};
  \end{scope}
\endtikzpicture \qquad
\tikzpicture
  \BPDStart{5}{5}
  \BPDPlace{1}{1}{\blank}
  \BPDPlace{1}{2}{\SE}
  \BPDPlace{1}{3}{\hor}
  \BPDPlace{1}{4}{\hor}
  \BPDPlace{1}{5}{\hor}
  \BPDPlace{2}{1}{\blank}
  \BPDPlace{2}{2}{\ver}
  \BPDPlace{2}{3}{\blank}
  \BPDPlace{2}{4}{\SE}
  \BPDPlace{2}{5}{\hor}
  \BPDPlace{3}{1}{\blank}
  \BPDPlace{3}{2}{\ver}
  \BPDPlace{3}{3}{\SE}
  \BPDPlace{3}{4}{\+}
  \BPDPlace{3}{5}{\hor}
  \BPDPlace{4}{1}{\SE}
  \BPDPlace{4}{2}{\+}
  \BPDPlace{4}{3}{\NW}
  \BPDPlace{4}{4}{\ver}
  \BPDPlace{4}{5}{\SE}
  \BPDPlace{5}{1}{\ver}
  \BPDPlace{5}{2}{\ver}
  \BPDPlace{5}{3}{\SE}
  \BPDPlace{5}{4}{\+}
  \BPDPlace{5}{5}{\+}
  \BPDEnd

  \begin{scope}[x=2em,y=2em, yscale=-1, shift={(0,5)}]
    \node[below] at (0.5,5) {\small 1};
    \node[below] at (1.5,5) {\small 2};
    \node[below] at (2.5,5) {\small 3};
    \node[below] at (3.5,5) {\small 4};
    \node[below] at (4.5,5) {\small 5};
  \end{scope}
  \begin{scope}[x=2em,y=2em, yscale=-1, shift={(0,5)}]
    \node[right] at (5,0.5) {\small 2};
    \node[right] at (5,1.5) {\small 4};
    \node[right] at (5,2.5) {\small 1};
    \node[right] at (5,3.5) {\small 5};
    \node[right] at (5,4.5) {\small 3};
  \end{scope}
\endtikzpicture \qquad 
\tikzpicture
  \BPDStart{5}{5}
  \BPDPlace{1}{1}{\blank}
  \BPDPlace{1}{2}{\blank}
  \BPDPlace{1}{3}{\SE}
  \BPDPlace{1}{4}{\hor}
  \BPDPlace{1}{5}{\hor}
  \BPDPlace{2}{1}{\blank}
  \BPDPlace{2}{2}{\blank}
  \BPDPlace{2}{3}{\ver}
  \BPDPlace{2}{4}{\SE}
  \BPDPlace{2}{5}{\hor}
  \BPDPlace{3}{1}{\blank}
  \BPDPlace{3}{2}{\SE}
  \BPDPlace{3}{3}{\+}
  \BPDPlace{3}{4}{\+}
  \BPDPlace{3}{5}{\hor}
  \BPDPlace{4}{1}{\SE}
  \BPDPlace{4}{2}{\+}
  \BPDPlace{4}{3}{\NW}
  \BPDPlace{4}{4}{\ver}
  \BPDPlace{4}{5}{\SE}
  \BPDPlace{5}{1}{\ver}
  \BPDPlace{5}{2}{\ver}
  \BPDPlace{5}{3}{\SE}
  \BPDPlace{5}{4}{\+}
  \BPDPlace{5}{5}{\+}
  \BPDEnd

  \begin{scope}[x=2em,y=2em, yscale=-1, shift={(0,5)}]
    \node[below] at (0.5,5) {\small 1};
    \node[below] at (1.5,5) {\small 2};
    \node[below] at (2.5,5) {\small 3};
    \node[below] at (3.5,5) {\small 4};
    \node[below] at (4.5,5) {\small 5};
  \end{scope}
  \begin{scope}[x=2em,y=2em, yscale=-1, shift={(0,5)}]
    \node[right] at (5,0.5) {\small 2};
    \node[right] at (5,1.5) {\small 4};
    \node[right] at (5,2.5) {\small 1};
    \node[right] at (5,3.5) {\small 5};
    \node[right] at (5,4.5) {\small 3};
  \end{scope}
\endtikzpicture
\end{gathered}
\]
\end{scriptsize}
    \end{example}

Let $\mathrm{blank}(B)$ be the set of coordinates carrying blank tiles~\bl\ and let $\mathrm{cross}(B)$ be the set of coordinates carrying cross tiles~\bpdplus. Unlike the classical model, a bumpless pipe dream is not determined solely by $\mathrm{blank}(B)$ or by $\mathrm{cross}(B)$. The \textit{monomial weight }of $B$ is the product
\[
x^{B}:=\prod_{(i,j)\in \mathrm{blank}(B)} x_i,
\]
and similarly \[
(x-y)^{B}:=\prod_{(i,j)\in \mathrm{blank}(B)} (x_i-y_j)
\]
with a second set of variables $y_j$. The following theorem is completely parallel to \Cref{thm:pipe dreams}.
\begin{theorem}[Lam-Lee-Shimozono \cite{LLS}]\label{thm:bumpless-schubert}
For $w\in S_n$,
\[
  \mathfrak{S}_w(x_1,\dots,x_n)=\sum_{B\in \operatorname{BPD}(w)} x^{B},
\] and \[
\mathfrak{S}_w(x_1,\dots, x_n\mid y_1,\dots, y_n) = \sum_{B\in \operatorname{BPD}(w)}(x-y)^B.
\]
\end{theorem}

It is a remarkable fact that both single and double Grothendieck polynomials can be recovered from bumpless pipe dreams by allowing non-reduced diagrams and modifying the weight. Concretely, let $\overline{\operatorname{BPD}}(w)$ denote the set of all not-necessarily-reduced BPDs of $w$. For $B\in \overline{\operatorname{BPD}}(w)$, let $\operatorname{U}(B)=\{(i,j)\in B\mid B\text{ has a upper-left elbow tile at }(i,j)\}$ and define the $K$-theoretic monomial weight to be \[\begin{gathered}\operatorname{wt}_{K}(B;x)\;=\;\prod_{(i,j)\in\mathrm{blank}(B)}-x_i\;\cdot\!\!\!\prod_{(i,j)\in \operatorname{U}(B)}\!\!\!(1-x_i), \\\operatorname{wt}_{K}(B;x,y)\;=\;\prod_{(i,j)\in\mathrm{blank}(B)}(-x_i-y_j+x_iy_j)\;\cdot\!\!\!\prod_{(i,j)\in \operatorname{U}(B)}\!\!\!(1-x_i - y_j+x_iy_j).\end{gathered}\]
Weigandt proved the following BPD--Grothendieck polynomial expansion.

\begin{theorem}[Weigandt \cite{weigandt2021bumpless}]\label{thm:BPD-grothendieck}
    For $w\in S_n$, its Grothendieck polynomial equals 
    \[\mathfrak{G}_w(x)\;=\;\sum_{B\in\operatorname{BPD}(w)}\operatorname{wt}_{K}(B;x).\]
    The double Grothendieck polynomial for $w$ equals \[\mathfrak{G}_w(x,y)\;=\;\sum_{B\in\operatorname{BPD}(w)}\operatorname{wt}_{K}(B;x,y).\]
\end{theorem}

There have since been fruitful developments to \Cref{thm:BPD-grothendieck}. In the back-stable setting, Lam-Lee-Shimozono gave $K$-bumpless models and expansion formulae for back-stable Grothendieck polynomials \cite[§6]{MR4681291}. Further developments include vexillary Grothendieck polynomials via BPD \cite{HafnerBPD}, and bijective frameworks unifying BPD-based formulas for Grothendieck polynomials \cite{HuangMarkedBPD}.

Similar to chute moves,  \textit{droop moves} are permutation-preserving local moves on BPD which generates all of $\overline{\operatorname{BPD}}(w)$, introduced by Lam-Lee-Shimozono \cite{LLS}. For $B\in \overline{\operatorname{BPD}}(w)$ that contains \bpdblank-tiles at $(i, j)$ and $(i+a, j+b)$, a \textit{droop move} replaces the ``$\lceil$-shape''$(i+a, j) \rightarrow(i, j) \rightarrow(i, j+b)$ by the ``$\rfloor$-shape'' $(i+a, j) \rightarrow(i+a, j+b) \rightarrow(i, j+b).$ Weigandt \cite{weigandt2021bumpless} introduced \textit{K-theoretic droop moves}, which apply to two pipes that are already crossing and performs the same replacement when the tile at $(i+a, j+b)$ is a SE-elbow. All of $\operatorname{BPD}(w)$ can be obtained from $B_{\textrm{diagram}}$ by performing droop moves, and all of $\overline{\operatorname{BPD}}(w)$ can be obtained from $B_{\textrm{diagram}}$ by $K$-theoretic droop moves. See \Cref{fig:droop-moves} for an illustration.
\begin{figure}[h]
    \centering
\begin{footnotesize}
    \[\tikzpicture
  \BPDStart{3}{3}
  \BPDPlace{1}{1}{\SE}
  \BPDPlace{1}{2}{\+}
  \BPDPlace{1}{3}{\hor}
  \BPDPlace{2}{1}{\+}
  \BPDPlace{2}{2}{\+}
  \BPDPlace{2}{3}{\hor}
  \BPDPlace{3}{1}{\ver}
  \BPDPlace{3}{2}{\ver}
  \BPDPlace{3}{3}{\blank}
  \BPDEnd
\endtikzpicture \quad \raisebox{3em}{$\mapsto$} \quad 
\tikzpicture
  \BPDStart{3}{3}
  \BPDPlace{1}{1}{\blank}
  \BPDPlace{1}{2}{\ver}
  \BPDPlace{1}{3}{\SE}
  \BPDPlace{2}{1}{\hor}
  \BPDPlace{2}{2}{\+}
  \BPDPlace{2}{3}{\+}
  \BPDPlace{3}{1}{\SE}
  \BPDPlace{3}{2}{\+}
  \BPDPlace{3}{3}{\NW}
  \BPDEnd
\endtikzpicture \qquad \qquad
\tikzpicture
  \BPDStart{3}{3}
  \BPDPlace{1}{1}{\SE}
  \BPDPlace{1}{2}{\+}
  \BPDPlace{1}{3}{\hor}
  \BPDPlace{2}{1}{\+}
  \BPDPlace{2}{2}{\+}
  \BPDPlace{2}{3}{\hor}
  \BPDPlace{3}{1}{\ver}
  \BPDPlace{3}{2}{\ver}
  \BPDPlace{3}{3}{\SE}
  \BPDEnd
\endtikzpicture \quad \raisebox{3em}{$\mapsto$} \quad 
\tikzpicture
  \BPDStart{3}{3}
  \BPDPlace{1}{1}{\blank}
  \BPDPlace{1}{2}{\ver}
  \BPDPlace{1}{3}{\SE}
  \BPDPlace{2}{1}{\hor}
  \BPDPlace{2}{2}{\+}
  \BPDPlace{2}{3}{\+}
  \BPDPlace{3}{1}{\SE}
  \BPDPlace{3}{2}{\+}
  \BPDPlace{3}{3}{\+}
  \BPDEnd
\endtikzpicture
\]
\end{footnotesize}
    \caption{Visualization of droop moves (left) and $K$-theoretic droop moves (right).}
    \label{fig:droop-moves}
\end{figure}

\subsection{The Variety of Spanning Line Configurations}
\label{sec:Xnk-cohomology}

In this section we review Pawlowski--Rhoades' computation of the cohomology ring of $\Xnk$. A nice summary of relevant constructions and results can be found in \cite{Rhoades2022GeneralizationsOfFlag}. We begin with the combinatorics of Fubini words, which index the cells in a natural stratification of $(\PP^{k-1})^n$ and $X_{n,k}$.

\label{subsec:Fubini-words}

\begin{definition}
Fix integers $1 \le k \le n$.  A \textit{word} of length $n$ with alphabet $[k]$ is a function $w:[n]\to[k]$.  We write its \textit{one-line notation} as the sequence $w_1 w_2 \dots w_n$, where $w_i := w(i)$.  Elements of $[n]$ are called \textit{positions} and elements of $[k]$ are called \textit{letters}.  The set of all such words is $[k]^n =: \Words$.  A word $w$ is a \textit{Fubini word} if it is surjective as a map $[n]\to[k]$, i.e., if its one-line notation contains every letter of $[k]$ at least once.  We denote the set of Fubini words by $\Fubini$.
\end{definition}

Fubini words are in bijection with \textit{ordered set partitions} of $[n]$ into $k$ nonempty blocks, via
\[
w \;\longmapsto\; \bigl(w^{-1}(1)\mid w^{-1}(2)\mid \cdots \mid w^{-1}(k)\bigr).
\]
In the special case $k = n$, Fubini words are precisely permutations in $\mathcal S_n$.  Thus Fubini words generalize permutations and inherit natural analogues of concepts such as inversions and Bruhat order \cite{BilleyRyan2024}.  The number of ordered set partitions of $[n]$ into $k$ blocks is the \textit{Fubini number} (\href{https://oeis.org/A000670}{OEIS~A000670}), and its value is
\[
\#\Fubini \;=\; k! \cdot \operatorname{Stir}(n,k),
\]
where $\operatorname{Stir}(n,k)$ denotes the Stirling number of the second kind (\href{https://oeis.org/A008277}{OEIS~A008277}).

\begin{definition}
    Let $w=w_1 \ldots w_n \in \Words$ be a word. A position $i\in [n]$ is an \textit{initial position} in $w$ if $w_h \neq w_i$ for all $h<i$. The letter appearing at an initial position is an \textit{initial letter}. Let $\operatorname{in}(w)=\{1 \leq$ $i \leq n\mid i$ is initial in $w\}$ be the set of all initial positions in $w$. Conversely, a position $i\in[n]$ is a \textit{redundant position} if there is some $h<i$ such that $w_h = w_i$. The letter at a redundant position is a \textit{redundant letter}. Let $\operatorname{re}(w)$ be the set of all redundant positions in $w$. 
\end{definition}
\begin{definition}\label{def:convexification}
    A word $w=w_1 \ldots w_n$ is called \textit{convex} if every copy of the letter $j \in[k]$ appears in consecutive positions. Given a word $w$, its \textit{convexification} $\conv(w)$ is the unique convex word with the same letters and the same ordering of initial letters. The \textit{associated permutation} $\sigma(w)$ to a convexification is the lex-minimal permutation of $[n]$ sending $\conv(w)$ back to $w$. 
\end{definition}
\begin{example}
\label{eg:convexification} The word $v = 2244433$ is convex, but $w=2442343$ is not, since the $2$'s are non-consecutive. The convexification of $w$ is $\conv(w) = 2244433.$ The initial positions of $w$ are $\operatorname{in}(w)=\{1,2,5\}$. The associated permutation is $\sigma(w) = 1423657,$ and we verify that $[2244433]\circ[1423657] = [2442343].$ 
\end{example}

\begin{definition}\label{def:standardization}
Let $w\in \Words$ be a word that uses $m$ letters out of $[k]$. Let $\left\{i_1<\cdots<i_{n-m}\right\} = [n]-\operatorname{in}(w)$ be the noninitial positions of $w$. Let  $\left\{j_1<\cdots<j_{k-m}\right\} = [k]-\left\{w_1, \ldots, w_n\right\}$ be the letters missing from $w$. The \textit{standardization} $\operatorname{std}(w)$ is a permutation in the symmetric group $ \mathcal{S}_{n+k-m}$ given as

$$
\operatorname{std}(w)_i=\left\{\begin{array}{ll}
w_i & \text { if } i \in \operatorname{in}(w) \\
k+r & \text { if } i=i_r \\
j_{i-n} & \text { if } i>n
\end{array} \text { for some } 1 \leq r \leq n-m. \right.
$$
\end{definition}
\begin{example}
Continuing \Cref{eg:convexification}, let $v $ be the convex word $2244433$. The standardization is $\operatorname{std}(v) = 25467381$. 
\end{example}

Given any word $w = w_1\dots w_n\in \Words$, we construct its \textit{pattern matrix} $\PM(w)$ to be the $k\times n$ matrix with entries $0,1,\star$ with entries $\PM(w)_{i,j}$ as in Algorithm~\ref{alg:pattern-matrix} in \Cref{apx:algorithms}.

\begin{example}
    Continuing with \Cref{eg:convexification}, the word $w =2442343$ has initial positions $1,2,5$ and pattern matrix $$
\operatorname{PM}(w)=\begin{pmatrix}
0 &0& 0& 0& 0& 0& 0\\
1 &\star &\star &1 &\star &\star &\star\\
0 &0 &0 &0 &1 &0 &1\\
0 &1 &1 &0 &0 &1 &\star\\
\end{pmatrix} .
$$
\end{example}

\begin{lemma}
  \cite[Observation 3.3]{PawlowskiRhoades2019} For $w\in \Words$, the number of $\star$'s in $\PM(w)$ is equal to $k\cdot n - \ell\left(\std(\conv(w))\right)$. 
\end{lemma}

\label{section:spanning-lines}

We now describe the cell decomposition of $\Xnk$ in terms of Fubini words and summarize the calculation of its integral cohomology $H^\bullet(\Xnk(\CC);\ZZ)$ following \cite{PawlowskiRhoades2019}. We first introduce some notation and refer the reader to \Cref{tab:words-and-matrices} for a summary thereof. Let $\Mat$ be the set of $k\times n$ matrices, viewed as the affine space $\kk^{k\times n}$. Let $\Nonzerocol$ be the subset of matrices, every column of which has at least one non-zero entry. Let $U$ be the group of lower unitriangular matrices of size $k\times k$ with $1$'s on the diagonal, and let $T = (\kk^\times )^n$ be the $n$-torus viewed as the group of $n\times n$ diagonal matrices with entries in $\kk^\times$. The affine space $\Mat$ has both a left $U$-action by elementary row operations and a right $T$-action by scaling each column, and this action restricts to $\Nonzerocol$. We may identify $\Pnk$ with the quotient $\Nonzerocol/T$. Each point $p= (p_1, \dots, p_n) \in \Pnk$ can be represented by a matrix $m\in \Nonzerocol$ whose columns are projective coordinates for $p_i\in \PP^{k-1}$ up to scaling by $T$. The \textit{Reduction Algorithm} (Algorithm 2 of \Cref{apx:algorithms}) of Pawlowski-Rhoades \cite[p9]{PawlowskiRhoades2019} leverages the $U\times T$-action to produce a canonical matrix representative of each point in $\Pnk$.

\begin{example}
    For the matrix $$m = \begin{pmatrix}
1 & 2 & 3 & 1 & 1 \\
2 & 1 & 3 & 0 & -1 \\
3 & -3 & 0 & 0 & 3
\end{pmatrix},$$ the Reduction Algorithm outputs $$\operatorname{R}(m) = \begin{pmatrix}
1 & -2 / 3 & -1 & 1 / 3 & 1 / 9 \\
0 & 1 & 1 & -2 / 3 & -1 / 3 \\
0 & 0 & 0 & 1 & 1
\end{pmatrix},\quad \operatorname{word}(m) = 12233. $$
\end{example}

The Reduction Algorithm outputs a unique canonical form regardless of the choice of the coset representative $mT$. Also recall that a point $p\in \Pnk$ corresponds to a coset of $T$. Therefore, we abuse notation and write $\operatorname{R}(p)$ in place of $\operatorname{R}(m)$ for $p\in \Pnk$. Similarly, we write $\operatorname{word}(p)$ for the other output of the algorithm regardless of the representative $mT$. 

    A matrix $m$ is said to \textit{fit the pattern} of a word $w$ if it agrees with $\PM(w)$ with the $\star$'s replaced by complex numbers. Let $P_w\subseteq \Nonzerocol$ be the set of matrices that fits the pattern of $w$. Generalizing \cite{BilleyRyan2024}, we make the following definition. 
    \begin{definition}
  For a word $w\in \Words$, the \textit{Pawlowski-Rhoades (PR) cell} $C_w$ consists of all points $p\in \Pnk$ for which $\operatorname{word}(p) = w$. Equivalently, it is the $U$-orbit $U P_w$ under the left unitriangular action. The \textit{Pawlowski-Rhoades variety} $X_w$ is the closure of $C_w$ inside $\Pnk$.
    \end{definition}

\begin{theorem}
    [Pawlowski--Rhoades \cite{PawlowskiRhoades2019} Theorem 5.12] \label{thm:PR-cell-decomposition} The $n$-fold product $\left(\mathbb{P}^{k-1}\right)^n$ admits a cell decomposition by PR cells $C_w$ for all words $w\in \Words$. The dimension of each cell $C_w$ is $$\binom{n}{2} + \text{ number of } \star\text{'s in } \PM(w) = n(k-1)-\ell\left(\std(\conv(w))\right).$$ 
\end{theorem}

We move on to describe the cell structure on $\Xnk$. Let $\FR$ be the subset of $\Nonzerocol$ with full column rank. Each point in $\Xnk$ can be represented by a matrix in $\FR$. Recall that $\Nonzerocol$ admits a left $U$-action and right $T$-action. The subset $\FR$ is closed under the $T$-action, and we may identify $\Xnk$ with the quotient $\FR/T$. The variety $\Xnk$ can be constructed from $\Pnk$ by removing the union of closures of cells represented by matrices with rank less than $k$.

\begin{theorem}
    [Pawlowski--Rhoades \cite{PawlowskiRhoades2019} Theorem 5.12]
    Let $C_w$ be the PR cells for the cell decomposition in \Cref{thm:PR-cell-decomposition}. The variety $X_{n, k}$ is the open set 

$$
X_{n, k} = \underset{w \in \Fubini}{\bigcup \quad C_w} =\left(\mathbb{P}^{k-1}\right)^n-\underset{w \in \Words - \Fubini}{\bigcup \quad X_w}
$$
in $\Pnk$ by removing the union of PR varieties indexed by non-Fubini words, and the PR cell decomposition restricts to a cell structure on $\Xnk$.
\label{thm:Xnk-cell-decomposition}
\end{theorem}

\begin{table}[ht]
    \centering
    \begin{tabular}{ccc} 
    \hline
         Words&  Matrices &Variety\\ 
         \hline
         $[0,k]^n$& 
     $\Mat$& $\bA^{k\times n}$\\ 
  $\Words$ & $\Nonzerocol$& $\Pnk$\\
  $\Fubini$&  $\FR$& $\Xnk$\\
  \hline
  \end{tabular}
    \caption{Correspondence between words and matrix representatives}
\label{tab:words-and-matrices}
\end{table}

The $i^\text{th}$ dual tautological line bundle over $\Pnk$ is the pullback of the line bundle $\cO_{\PP^k}(1)$ along the $i^\text{th}$ coordinate projection. Since $\Xnk$ is an open subset of $\Pnk$, these dual tautological line bundles pull back to $\Xnk$ along the open immersion. Let $x_i$ denote the first Chern class of the $i^\text{th}$ dual tautological line bundle over $\Xnk$. Pawlowski-Rhoades derives a presentation of the cohomology ring $H^\bullet(\Xnk(\CC);\ZZ)$ by identifying the closures of the PR-cells with degeneracy loci of flagged vector bundles. The proof then follows from the degeneracy loci formula for Schubert polynomials (\Cref{thm:Schubert-degeneracy-loci-formula}) and a rank argument.

\begin{definition}
   \label{def:schubert-poly-of-word}
    Pawlowski-Rhoades \cite{PawlowskiRhoades2019} defines the \textit{Schubert polynomial} of a word $w\in \Words$ to be \begin{equation}
        \label{eqn:schubert-poly-of-word}
        \mathfrak{S}_w(x_1, \dots, x_n) := \sigma^{-1}(w)\cdot \mathfrak{S}_{\operatorname{std}(\conv(w))}(x_1, \dots, x_n)
    \end{equation}
    where $\sinv \in S_n$ is the associated permutation of $w$ acting by permuting the variables.
\end{definition}

\begin{theorem}[Pawlowski--Rhoades \cite{PawlowskiRhoades2019} Theorem 5.12]
    \label{thm:H*(Xnk)}
There is an isomorphism \begin{equation}\label{eq:H*(Xnk)}
    H^\bullet(\Xnk(\CC); \ZZ) \cong \Rnk = \frac{\ZZ[x_1, x_2, \dots, x_n]}{(x_1^k, \dots, x_n^k, e_{n-k+1}, \dots, e_n)}.
\end{equation}
\end{theorem}

\begin{theorem}[Proposition 5.9, 5.11 \cite{PawlowskiRhoades2019}]
    \label{thm:schubert-for-word} Let $w\in\Words$ be a word and $C_w$ be its Pawlowski-Rhoades cell. Under the isomorphism \eqref{eq:H*(Xnk)}, the cohomology class $[X_w]$ is mapped to the Schubert polynomial $\mathfrak{S}_w(x_1, \dots, x_n)$.  Furthermore, $\{\mathfrak{S}_w\mid w\in \Fubini\}$ is an additive basis of $H^\bullet(\Xnk(\CC);\ZZ)\cong \Rnk$.  \end{theorem}

\section{The Grothendieck Group of Spanning Line Configurations}
\label{sec:K0-of-Xnk}

In this section we identify the Grothendieck group $K_0(X_{n,k})$ with the generalized coinvariant algebra $R_{n,k}$ and complete the proof of our main theorem.  The argument has three ingredients.  First, we prove some general facts about $G_0$ and $K_0$ of cellular varieties, including a localization statement for complements of unions of cell closures (\Cref{thm:K0(X-Z)=K0(X)/I(Z)}).  Second, we identify the $K$-theory classes of Pawlowski--Rhoades varieties $X_w$ inside $K_0((\PP^{k-1})^n)$ with Grothendieck polynomials of words (\Cref{thm:grothendieck-poly-of-word}).  Finally, we compare the resulting quotient to the known presentation of $R_{n,k}$ and use a rank argument to conclude the proof of \Cref{thm:main}.

\subsection{Key lemmas for cellular varieties}
\label{subsec:key-lemmas-for-cellular}

We begin by extending the cellular freeness statement \Cref{thm:cellular-finite-generation-of-cohomology} from homology and Chow groups to the zeroth $G$-theory.

\begin{theorem}
\label{thm:G0-free-abelian-when-cellular}
Let $X$ be a (possibly singular) cellular variety with cells $\{C_i\mid i\in \mathcal{I}\}$ indexed by a finite set $\mathcal{I}$.  Then $G_0(X)$ is a free abelian group generated by the classes of the structure sheaves of cell closures $\{[\cO_{\overline{C}_i}]\mid i\in \mathcal I\}$.
\end{theorem}

\begin{proof}
We proceed by induction on the number of cells $n := |\mathcal I|$.  If $X$ has a single cell, then $X \cong \bA^r$ for some $r$, and homotopy invariance (\Cref{thm:K0-homotopy}) gives $G_0(\bA^r) \cong \ZZ\cdot[\cO_{\bA^r}]$, so the statement holds.

Now assume the result for all cellular varieties with fewer than $n$ cells, and let $X$ have $n$ cells.  We first show that $X$ contains at least one closed cell.  Pick a cell $C_0$ of minimal dimension, and suppose for contradiction that $C_0$ is not closed.  Then there exists a point $x\in \overline{C_0} - C_0$, which lies in some other cell $C_1$ since $X$ is the disjoint union of its cells.  The closure $\overline{C_0}$ is irreducible and locally closed, and any irreducible subset properly contained in it must have smaller dimension, so $\dim C_1 < \dim C_0$, contradicting the minimality of $\dim C_0$.  Hence $C_0$ is closed in $X$.

Let $i:C_0\hookrightarrow X$ and $j:U:=X- C_0\hookrightarrow X$ be the inclusions.  As $i$ is a closed immersion, the localization sequence in $G$-theory (\Cref{thm:K0-localization}) yields a right exact sequence
\[
G_0(C_0) \xrightarrow{i_*} G_0(X) \xrightarrow{j^*} G_0(U) \longrightarrow 0.
\]
Since $C_0\cong \bA^{d_0}$ is affine, homotopy invariance gives $G_0(C_0) \cong \ZZ\cdot[\cO_{C_0}]$.  The restriction of the cell decomposition to $U$ has cells $C_i$ for $i\in \mathcal I-\{0\}$, so by the induction hypothesis $G_0(U)$ is freely generated by $\{[\cO_{\overline{C}_i}]\mid i\neq 0\}$ and hence $G_0(U)\cong \ZZ^{\oplus (n-1)}$.  Thus $G_0(X)$ fits into a short exact sequence of abelian groups
\[
0\longrightarrow \ZZ \xrightarrow{i_*} G_0(X) \xrightarrow{j^*} \ZZ^{\oplus (n-1)} \longrightarrow 0,
\]
so $G_0(X) \cong \ZZ^{\oplus (n-1)}\oplus A$ for some abelian group $A$.

To determine $A$, we compare ranks after tensoring with $\QQ$.  Since $X$ is cellular, its homology is free abelian on the closures of its cells, so $H_\bullet(X)\cong \ZZ^{\oplus n}$.  Baum-Fulton-MacPherson \cite[\textsc{III.I}]{BFM1975} constructed a Riemann-Roch isomorphism $\tau: G_0(X)\otimes \QQ \xrightarrow{\;\cong\;} CH_\bullet(X)\otimes \QQ$ for arbitrary varieties, and for cellular $X$ the cycle class map identifies $CH_\bullet(X)\otimes \QQ$ with $H_\bullet(X)\otimes \QQ \cong \QQ^{\oplus n}$ (\Cref{thm:cellular-finite-generation-of-cohomology}).  Hence $G_0(X)\otimes \QQ \cong \QQ^{\oplus n}$, which forces $A\otimes\QQ\cong \QQ$.  The only possibility compatible with the above exact sequence is $A\cong \ZZ$.  Thus $G_0(X)$ is a free abelian group of rank $n$, generated by the $n$ classes $[\cO_{\overline{C}_i}]$, as claimed.
\end{proof}

\begin{corollary}
\label{cor:G0-localization-exact}
Let $X$ be a cellular variety and $Z\subseteq X$ a union of cell closures.  Then the localization sequence
\[
0\longrightarrow G_0(Z) \xrightarrow{i_*} G_0(X) \xrightarrow{j^*} G_0(X- Z) \longrightarrow 0
\]
is exact.
\end{corollary}

\begin{proof}
Apply the argument in the proof of \Cref{thm:G0-free-abelian-when-cellular} to the pair $(X,Z)$ and note that $G_0(Z)$ and $G_0(X)$ are free abelian with bases given by the structure sheaves of cell closures.  The right exactness comes from \Cref{thm:K0-localization}. The rank argument shows that the map $G_0(Z)\to G_0(X)$ must be injective.
\end{proof}

We now prove the quotient description of $K_0$ for complements of unions of cell closures.

\begin{proof}[Proof of \Cref{thm:K0(X-Z)=K0(X)/I(Z)}]
Let $X$ be smooth and cellular, and let $Z\subseteq X$ be a union of cell closures.  By \Cref{thm:G0-free-abelian-when-cellular}, both $G_0(Z)$ and $G_0(X)$ are free abelian with bases given by the classes $[\cO_{\overline{C}}]$ of cell closures in $Z$ and in $X$, respectively.  Since $X$ and $X- Z$ are smooth, we may identify $K_0(X)\cong G_0(X)$ and $K_0(X- Z)\cong G_0(X- Z)$.

By \Cref{cor:G0-localization-exact} we have a short exact sequence
\[
0 \longrightarrow G_0(Z) \xrightarrow{i_*} K_0(X) \xrightarrow{j^*} K_0(X- Z) \longrightarrow 0.
\]
The map $j^*$ is a ring homomorphism by functoriality, and its kernel is generated as an abelian group by the image of $G_0(Z)$, which is the subgroup of $K_0(X)$ generated by the fundamental classes $[\cO_{\overline{C}}]$ of all cell closures $\overline{C}\subseteq Z$.  By definition, this subgroup is the ideal $I(Z)$, so the exact sequence identifies
\[
K_0(X- Z) \;\cong\; K_0(X)/I(Z)
\]
as rings.
\end{proof}

\begin{example}
Let $X = \PP^n$ with the standard cell decomposition $\PP^n = \coprod_{k=0}^n \bA^k$ as in \Cref{eg:K0(Pn)}, and let $Z = \PP^i = \overline{\bA^i}$ be the closure of the open cell of dimension $i$.  In $K_0(\PP^n)\cong \ZZ[\eta]/(\eta^{n+1})$, the class of $\PP^i$ is $\eta^{n-i}$, so \Cref{thm:K0(X-Z)=K0(X)/I(Z)} gives
\[
K_0(\PP^n- \PP^i) \;\cong\; \frac{\ZZ[\eta]/(\eta^{n+1})}{(\eta^{n-i})} \;\cong\; \frac{\ZZ[\eta]}{(\eta^{n-i})}.
\]
On the other hand, $\PP^n- \PP^i$ is an affine bundle over the complementary linear subspace $\PP^{n-i-1}$, and homotopy invariance implies $K_0(\PP^n- \PP^i)\cong K_0(\PP^{n-i-1})\cong \ZZ[\eta]/(\eta^{n-i})$, in agreement with the above quotient description.
\end{example}

\subsection{The Grothendieck polynomial of a word}

Recall from \Cref{thm:Xnk-cell-decomposition} that $X_{n,k}$ is obtained from $\Pnk$ by removing the Pawlowski--Rhoades varieties $X_w$ for non-Fubini words $w\in \Words- \Fubini$.  By \Cref{thm:K0(X-Z)=K0(X)/I(Z)}, the Grothendieck group $K_0(X_{n,k})$ is isomorphic to $K_0(\Pnk)$ modulo the ideal generated by the classes $[\cO_{X_w}]$ for $w\notin\Fubini$.  In this subsection we identify these classes with Grothendieck polynomials of words as defined in \Cref{def:grothendieck-poly-of-word}.

\begin{proof}[Proof of \Cref{thm:grothendieck-poly-of-word}]
We follow the strategy of Pawlowski--Rhoades \cite{PawlowskiRhoades2019}.  First suppose that $w\in \Words$ is a convex word.  Let $\mathcal{L}_1,\dots,\mathcal{L}_n$ denote the tautological line bundles on $\Pnk = (\PP^{k-1})^n$, and for each $j\in [n]$ set
\[
\mathcal E_j := \mathcal{L}_1 \oplus \cdots \oplus \mathcal{L}_j.
\]
For each $i\in [k]$, let $\mathcal F_i$ be the trivial rank-$i$ vector bundle equipped with the natural surjection $\mathcal F_i\to \mathcal F_{i-1}$ given by projection onto the first $i-1$ coordinates, and set $\mathcal F_0 := 0$.  Consider the bundle map $f:\mathcal E_n \to \mathcal F_k$ defined on fibers by
\[
f\bigl((v_1,\dots,v_n)\bigr) := v_1+\cdots+v_n \in \kk^k.
\]
For each pair $(i,j)$, the composition
\[
\mathcal E_j \hookrightarrow \mathcal E_n \xrightarrow{f} \mathcal F_k \twoheadrightarrow \mathcal F_i
\]
records the span of the first $j$ lines inside $\kk^k$.  As explained in \cite{PawlowskiRhoades2019}, the rank conditions determined by the convex word $w$ cut out a flagged degeneracy locus $\Omega_w(f)$ that coincides with the PR variety $X_w$, and the codimension of $X_w$ equals the length of the permutation $\std(w)$.

Since $X_w$ has the expected codimension, the Fulton-Lascoux degeneracy loci formula (\Cref{thm:Grothendieck-degeneracy-loci-formula}) implies that in $K_0(\Pnk)$ we have
\[
[\cO_{X_w}]
= [\cO_{\Omega_w(f)}]
= \mathfrak{G}_{\std(w)^{-1}}(a_1^K,\dots,a_k^K \mid b_1^K,\dots,b_n^K),
\]
where
\[
a_i^K = c_1^K\bigl(\ker(\mathcal F_i\to \mathcal F_{i-1})\bigr), \qquad
b_j^K = c_1^K(\mathcal E_j/\mathcal E_{j-1}).
\]
In our situation each $\mathcal F_i$ is trivial, so $a_i^K = 0$ for all $i$.  On the other hand, in the quotient presentation
\[
K_0(\Pnk) \cong \ZZ[x_1,\dots,x_n]/(x_1^k,\dots,x_n^k)
\]
the variable $x_j$ is $c_1^K(\mathcal{L}_j^\vee) = c_1^K\bigl((\mathcal E_j/\mathcal E_{j-1})^\vee\bigr)$, so $b_j^K = -x_j$.  It follows that
\[
[\cO_{X_w}] = \mathfrak{G}_{\std(w)^{-1}}(0,\dots,0 \mid -x_1,\dots,-x_n).
\]
By the basic symmetry of Grothendieck polynomials (see, e.g., \cite[Lemma~5.3]{MR4681291}), we have
\[
\mathfrak{G}_{\std(w)^{-1}}(0 \mid -x_1,\dots,-x_n)
= \mathfrak{G}_{\std(w)}(x_1,\dots,x_n),
\]
so for convex $w$ the class $[\cO_{X_w}]$ is represented by the ordinary Grothendieck polynomial $\mathfrak{G}_{\std(w)}$.

For a general word $w\in\Words$, we pass to its convexification $\conv(w)$ and use the permutation $\sigma(w)\in S_n$ sending $\conv(w)$ back to $w$, as in \Cref{def:grothendieck-poly-of-word}.  Pawlowski-Rhoades \cite[Proposition~5.11]{PawlowskiRhoades2019} show, by induction on the length of $\std(w)$, that $X_w$ is obtained from $X_{\conv(w)}$ by a sequence of geometric operations compatible with the natural $S_n$-action permuting the variables $x_i$ in $K_0(\Pnk)$.  The same inductive argument, together with the convex case treated above, implies that
\[
[\cO_{X_w}] = \sigma(w)^{-1}\cdot \mathfrak{G}_{\std(\conv(w))}(x_1,\dots,x_n).
\]
By \Cref{def:grothendieck-poly-of-word} this is precisely the Grothendieck polynomial $\mathfrak{G}_w(x_1,\dots,x_n)$, which completes the proof.
\end{proof}

\subsection{The Grothendieck group of spanning line configurations}

We now assemble the results of the previous subsections to prove \Cref{thm:main}.  The cohomological part of the statement is due to Pawlowski--Rhoades and was recalled in \Cref{thm:H*(Xnk)}. The new content here is the identification of $K_0(X_{n,k})$ with $R_{n,k}$ and the compatibility of this identification with Chern classes of tautological line bundles.

\begin{proof}[Proof of \Cref{thm:main}]
We begin with the cohomological and Chow-theoretic identifications.  Since $X_{n,k}$ is smooth and cellular, the cycle class map
\[
cl : CH^\bullet(X_{n,k}) \longrightarrow H^\bullet(X_{n,k}(\CC);\ZZ)
\]
is a degree-doubling isomorphism of graded rings (\Cref{thm:cycle-class-map}).  Pawlowski-Rhoades \cite[Thm.~5.12]{PawlowskiRhoades2019} identify $H^\bullet(X_{n,k}(\CC);\ZZ)$ with the generalized coinvariant algebra $R_{n,k}$, so we obtain isomorphisms of graded rings
\[
CH^\bullet(X_{n,k}) \;\cong\; H^\bullet(X_{n,k}(\CC);\ZZ) \;\cong\; R_{n,k}.
\]

Next we construct a surjection from $R_{n,k}$ onto $K_0(X_{n,k})$.  The variety $X_{n,k}$ is obtained from $\Pnk$ by removing the PR varieties $X_w$ for non-Fubini words $w\in \Words- \Fubini$.  Applying \Cref{thm:K0(X-Z)=K0(X)/I(Z)} with $X=\Pnk$ and $Z = \bigcup_{w\notin \Fubini} X_w$ shows that
\[
K_0(X_{n,k}) \;\cong\; \frac{K_0(\Pnk)}{I([\cO_{X_w}] \mid w\notin \Fubini)}.
\]
Using the product presentation $K_0(\Pnk)\cong \ZZ[x_1,\dots,x_n]/(x_1^k,\dots,x_n^k)$ (\Cref{eg:K0(Pn)}), and the identification of $[\cO_{X_w}]$ with the Grothendieck polynomial $\mathfrak{G}_w$ from \Cref{thm:grothendieck-poly-of-word}, we see that $K_0(X_{n,k})$ is the quotient of
\[
S_{n,k} := \ZZ[x_1,\dots,x_n]/(x_1^k,\dots,x_n^k)
\]
by the ideal generated by $\{\mathfrak{G}_w \mid w\notin\Fubini\}$.

On the other hand, the generalized coinvariant algebra $R_{n,k}$ is defined as
\[
R_{n,k} = \frac{\ZZ[x_1,\dots,x_n]}{(x_1^k,\dots,x_n^k,e_{n-k+1},\dots,e_n)}.
\]
Pawlowski-Rhoades \cite{PawlowskiRhoades2019} showed that the Schubert polynomials of the special words $w^{(i)} := 12\cdots \widehat{i}\cdots kk\cdots k$ map to alternating sums of elementary symmetric polynomials $e_{n-i+1},\dots,e_n$.  In our $K$-theoretic setting, the corresponding PR varieties $X_{w^{(i)}}$ have classes
\[
[\cO_{X_{w^{(i)}}}] \mapsto \mathfrak{G}_{v^{(i)}}(x_1,\dots,x_n)
\]
where $v^{(i)} := 12\cdots \widehat{i}\cdots ni$ is the associated Grassmannian permutation and $\mathfrak{G}_{v^{(i)}}$ is an alternating sum of $e_j$ for $n+1-i\le j\le n$ with leading term $e_{n-i+1}$; see \Cref{eg:grassmannian-grothendieck}.  Writing
\[
\mathfrak{G}_{v^{(i)}} = e_{n-i+1} + \sum_{n+1-i<j\le n} (-1)^{i+j-n-1} g_j^i e_j,
\]
the change of generators from $\{e_{n-k+1},\dots,e_n\}$ to $\{\mathfrak{G}_{v^{(n-k+1)}},\dots,\mathfrak{G}_{v^{(n)}}\}$ is unitriangular: starting from $i=n$, we have $\mathfrak{G}_{v^{(n)}} = e_n$, and at each lower index $i$ one can recover $e_{n-i+1}$ from $\mathfrak{G}_{v^{(i)}}$ by subtracting a linear combination of the already-present higher $e_j$’s.  In particular, the higher-order terms in the $\mathfrak{G}_{v^{(i)}}$ cancel inductively inside the ideal, and the ideal generated by the $\mathfrak{G}_{v^{(i)}}$ coincides with the ideal generated by $e_{n-k+1},\dots,e_n$.  Since each $w^{(i)}$ is non-Fubini, the classes $[\cO_{X_{w^{(i)}}}]$ lie in the ideal $I([\cO_{X_w}]\mid w\notin\Fubini)$, and hence
\[
I_{n,k} := (x_1^k,\dots,x_n^k,e_{n-k+1},\dots,e_n)
\]
is contained in the kernel of the natural map $\ZZ[x_1,\dots,x_n]\to K_0(X_{n,k})$.  This yields a surjective ring homomorphism
\[
R_{n,k} = \frac{\ZZ[x_1,\dots,x_n]}{I_{n,k}} \;\twoheadrightarrow\; K_0(X_{n,k}).
\]

Finally, we compare ranks.  Since $X_{n,k}$ is cellular, \Cref{thm:G0-free-abelian-when-cellular} implies that $K_0(X_{n,k})$ is a free abelian group with basis given by the classes $[\cO_{X_w}]$ for Fubini words $w\in \Fubini$, so
\[
\operatorname{rank}_\ZZ K_0(X_{n,k}) = \#\Fubini = k!\,\operatorname{Stir}(n,k).
\]
Pawlowski-Rhoades \cite[Thm.~5.12]{PawlowskiRhoades2019} show that $R_{n,k}$ has the same rank as a free abelian group.  Thus the surjection $R_{n,k}\twoheadrightarrow K_0(X_{n,k})$ is an isomorphism of rings.  Combining this with the cohomological identifications above, we obtain
\[
K_0(X_{n,k}) \;\cong\; CH^\bullet(X_{n,k}) \;\cong\; H^\bullet(X_{n,k}(\CC);\ZZ) \;\cong\; R_{n,k},
\]
as claimed.  The description of the isomorphism $K_0(X_{n,k})\cong R_{n,k}$ in terms of $K$-theoretic first Chern classes of the dual tautological line bundles follows directly from the presentation of $K_0((\PP^{k-1})^n)$ and the restriction to $X_{n,k}$.
\end{proof}

\begin{corollary}
\label{cor:grothendieck-additive-basis-Xnk}
The set $\{\mathfrak{G}_w \mid w\in \Fubini\}$ forms an additive $\ZZ$-basis of $K_0(X_{n,k})\cong R_{n,k}$. \hfill $\square$
\end{corollary}

\section{Pipe Dream Combinatorics of Words}
\label{sec:PD-for-words}

In this section, we extend the classical and bumpless pipe dream models from permutations to arbitrary words $w\in [k]^n$.  The resulting pipe dreams for words will realize the Schubert and Grothendieck polynomials $\mathfrak S_w$ and $\mathfrak G_w$ as monomial-weight generating functions, in complete analogy with the permutation case.  We adopt the same notation as in \Cref{subsec:classical-PD}.

Recall from \Cref{thm:schubert-for-word} and \Cref{thm:grothendieck-poly-of-word} that for a word $w$ we define
\[\begin{gathered}
\mathfrak{S}_w(x_1,\dots,x_n) := \sinv\cdot \mathfrak{S}_{\std(\conv(w))}(x_1,\dots,x_n)
\quad\text{and} \\ \quad
\mathfrak{G}_w(x_1,\dots,x_n) := \sinv\cdot \mathfrak{G}_{\std(\conv(w))}(x_1,\dots,x_n),
\end{gathered}
\]
where $\conv(w)$ is the convexification of $w$, $\std(\conv(w))$ is its standardization, and $\sinv\in S_n$ is the permutation associated to $w$ acting by permuting the variables.  Thus it is natural to construct pipe dreams and bumpless pipe dreams for $w$ by starting from those for the permutation $\std(\conv(w))$ and then truncating and relabeling.  We first record a “rectangularity” property of these permutation pipe dreams that justifies the truncation to the leftmost $k$ columns.

\begin{lemma}\label{lem:rectangular-PD}
    Let $P$ be a classical pipe dream for the permutation $\std(\conv(w))$. Then, all cross-tiles of $P$ are in the leftmost $k$ columns.
\end{lemma}
\begin{proof}
 We first prove the statement for the top pipe dream $P_{\textrm{top}}\in \operatorname{PD}(\std(\conv(w)))$. Recall that $P_{\textrm{top}}$ is the unique pipe dream obtained by placing $c(\std(\conv(w))^{-1})_i$-many top-justified \crosstile-tiles at column $i$ of the grid. The Lehmer code $c(\std(\conv(w))^{-1})$ is determined by the position $j$ for each $(i,j)\in \operatorname{Inv}(\std(\conv(w))^{-1})$, which is in bijection with $(w_k, w_h)$ for $(h,k)\in \operatorname{Inv}(\std(\conv(w)))$. Therefore, the rightmost column that contains a \crosstile-tile in $P_{\textrm{top}}$ is the largest \textit{letter} $w_h$ in all inversions $(h,k)\in \operatorname{Inv}(\std(\conv(w)))$.
 
 The permutation $\std(\conv(w))$ consists of three kinds of letters -- initial letters, redundant letters, and missing letters. Let $m$ be the largest initial letter. We show that $m$ is the largest letter appearing as the second letter in any inversion of $\std(\conv(w))$. Redundant letters are by definition larger than $m$, so they cannot appear as the second letter in any inversion. Missing letters are by definition less than $m$, so they cannot be largest among second letters in inversions. Therefore, the above claim is true, and all \crosstile-tiles of $P_{\textrm{top}}$ must appear within the leftmost $k$ columns. 

 To see that this is true for \textit{every} pipe dream $P$ of $\std(\conv(w))$, recall that $P$ can be obtained from $P_{\textrm{top}}$ by a sequence of ($K$-theoretic) chute moves, and chute moves shift a \crosstile-tile always to the \textit{left}. 
\end{proof}

\begin{lemma}\label{lem:rectangular-BPD}
    Let $B$ be a bumpless pipe dream for the permutation $\std(\conv(w))$. Then, all blank tiles of $B$ are in the leftmost $k$ columns.
\end{lemma}
\begin{proof}
Similar to the proof of \Cref{lem:rectangular-PD}, every blank tile in the diagram BPD $B_{\textrm{diagram}}$ is contained in the leftmost $k$ columns. Recall that every BPD can be obtained from $B_{\textrm{diagram}}$ via a sequence of ($K$-theoretic) droop moves, and droop moves only shift a blank tile to the left. Therefore, all blank tiles of every BPD of $\std(\conv(w))$ are forced to be in the leftmost $k$ columns.
\end{proof}
As a result of the above lemmas, we may ignore the rightmost $(n-k)$ columns in both cases without causing ambiguity. It is therefore natural for us to define regular and bumpless pipe dreams for words by truncating the diagrams for $u:=\std(\conv(w))$ to the first $k$ columns and then relabeling rows by $\sinv$. Our definition of bumpless pipe dreams of words is compatible with the definition of the diagram of a word by Billey-Ryan \cite{BilleyRyan2024}.

\begin{definition}\label{def:PD-for-word}
    We define a \textit{classical pipe dream} $P$ for the word $w\in\Words$ to be the diagram consisting of the first $k$ columns of a classical pipe dream of the permutation $\std(\conv(w))$, with row labels given by the values of $\sinv$. We say the pipe dream $P$ is reduced if it comes from a reduced pipe dream of  $\std(\conv(w))$.
\end{definition}

\begin{definition}
\label{def:BPD-for-word}    We define a \textit{bumpless pipe dream} $B$ for the word $w$ to be the $n\times k$ grid consisting of the first $k$ columns of a bumpless pipe dream of the permutation $\std(\conv(w))$, with row labels given by the values of $\sinv$. We say the bumpless pipe dream $B$ is reduced if it comes from a reduced BPD of  $\std(\conv(w))$.
\end{definition}

For a classical pipe dream $P$ (or bumpless pipe dream $B$) for a word $w$, we define its monomial weight using the same formulas as in \Cref{subsec:classical-PD}, applied to the truncated diagram with relabeled rows.  In particular, the Schubert weight of $P$ is
\[
x^P := \prod_{(i,j)\in P} x_i,
\]
and similarly we define $(x-y)^P$ and the $K$-theoretic weights $\operatorname{wt}_K(B;x)$ and $\operatorname{wt}_K(B;x,y)$ by the same recipes as in the permutation case.

With these conventions, the local moves introduced earlier descend naturally.  Any chute move or $K$-theoretic chute move on a pipe dream for $u$ that is supported entirely in the first $k$ columns induces a well-defined move on the corresponding pipe dream for $w$, and similarly droop moves and $K$-theoretic droop moves for bumpless pipe dreams restrict to moves on word BPDs.  These induced moves do not introduce crosses or blanks in the truncated columns, so the combinatorics of chute and droop moves remains entirely valid for word pipe dreams associated to $w$.

\begin{proposition}
For a word $w\in [k]^n$ and $u := \std(\conv(w))$, every reduced pipe dream (resp.\ bumpless pipe dream) for $w$ can be obtained from the top (resp.\ diagram) pipe dream for $u$ by a sequence of chute moves (resp.\ droop moves) supported in the first $k$ columns, followed by truncation and relabeling.  In particular, the induced chute and droop moves act transitively on the reduced pipe dreams and reduced BPDs of $w$.
\end{proposition}

\begin{proof}
Follows from the proofs of \Cref{lem:rectangular-PD,lem:rectangular-BPD}.
\end{proof}

\begin{example}[Classical and bumpless pipe dreams for $w=21231$]
\label{eg:BPD-word-small}
Take $w=21231$ with $n=5$ and $k=3$. Then $\std(\conv(w))=u=24153$ and $\sinv=13254$. Below are all reduced pipe dreams of $w$. 

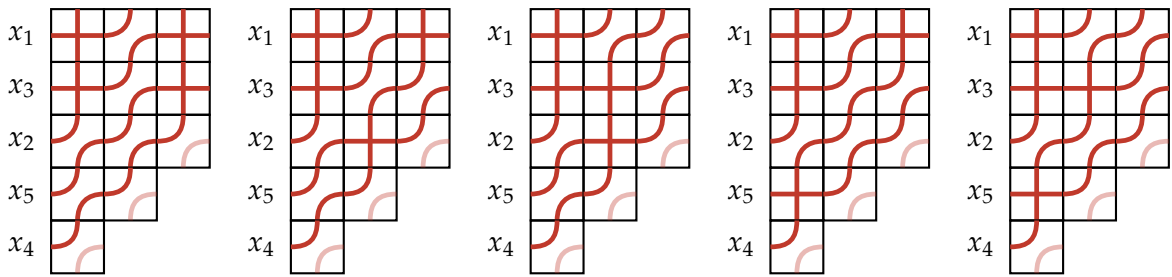
\begin{footnotesize}
    \[\tikzpicture
  \PipeStart{5}
  \PipePlace{1}{1}{\pipeCross}
  \PipePlace{1}{2}{\pipeBump}
  \PipePlace{1}{3}{\pipeCross}
  \PipePlace{2}{1}{\pipeCross}
  \PipePlace{2}{2}{\pipeBump}
  \PipePlace{2}{3}{\pipeCross}
  \PipePlace{3}{1}{\pipeBump}
  \PipePlace{3}{2}{\pipeBump}
  \PipePlace{3}{3}{\pipeElbow}
  \PipePlace{4}{1}{\pipeBump}
  \PipePlace{4}{2}{\pipeElbow}
  \PipePlace{5}{1}{\pipeElbow}
  \PipeEnd
  \begin{scope}[x=2em,y=2em, yscale=-1, shift={(-6,5)}]
    \node[right] at (5,0.5) {\small $x_1$};
    \node[right] at (5,1.5) {\small $x_3$};
    \node[right] at (5,2.5) {\small $x_2$};
    \node[right] at (5,3.5) {\small $x_5$};
    \node[right] at (5,4.5) {\small $x_4$};
  \end{scope}
\endtikzpicture \quad
\tikzpicture
  \PipeStart{5}
  \PipePlace{1}{1}{\pipeCross}
  \PipePlace{1}{2}{\pipeBump}
  \PipePlace{1}{3}{\pipeCross}
  \PipePlace{2}{1}{\pipeCross}
  \PipePlace{2}{2}{\pipeBump}
  \PipePlace{2}{3}{\pipeBump}
  \PipePlace{3}{1}{\pipeBump}
  \PipePlace{3}{2}{\pipeCross}
  \PipePlace{3}{3}{\pipeElbow}
  \PipePlace{4}{1}{\pipeBump}
  \PipePlace{4}{2}{\pipeElbow}
  \PipePlace{5}{1}{\pipeElbow}
  \PipeEnd
  \begin{scope}[x=2em,y=2em, yscale=-1, shift={(-6,5)}]
    \node[right] at (5,0.5) {\small $x_1$};
    \node[right] at (5,1.5) {\small $x_3$};
    \node[right] at (5,2.5) {\small $x_2$};
    \node[right] at (5,3.5) {\small $x_5$};
    \node[right] at (5,4.5) {\small $x_4$};
  \end{scope}
\endtikzpicture
\quad
\tikzpicture
  \PipeStart{5}
  \PipePlace{1}{1}{\pipeCross}
  \PipePlace{1}{2}{\pipeBump}
  \PipePlace{1}{3}{\pipeBump}
  \PipePlace{2}{1}{\pipeCross}
  \PipePlace{2}{2}{\pipeCross}
  \PipePlace{2}{3}{\pipeBump}
  \PipePlace{3}{1}{\pipeBump}
  \PipePlace{3}{2}{\pipeCross}
  \PipePlace{3}{3}{\pipeElbow}
  \PipePlace{4}{1}{\pipeBump}
  \PipePlace{4}{2}{\pipeElbow}
  \PipePlace{5}{1}{\pipeElbow}
  \PipeEnd
  \begin{scope}[x=2em,y=2em, yscale=-1, shift={(-6,5)}]
    \node[right] at (5,0.5) {\small $x_1$};
    \node[right] at (5,1.5) {\small $x_3$};
    \node[right] at (5,2.5) {\small $x_2$};
    \node[right] at (5,3.5) {\small $x_5$};
    \node[right] at (5,4.5) {\small $x_4$};
  \end{scope}
\endtikzpicture
\quad 
\tikzpicture
  \PipeStart{5}
  \PipePlace{1}{1}{\pipeCross}
  \PipePlace{1}{2}{\pipeBump}
  \PipePlace{1}{3}{\pipeCross}
  \PipePlace{2}{1}{\pipeCross}
  \PipePlace{2}{2}{\pipeBump}
  \PipePlace{2}{3}{\pipeBump}
  \PipePlace{3}{1}{\pipeBump}
  \PipePlace{3}{2}{\pipeBump}
  \PipePlace{3}{3}{\pipeElbow}
  \PipePlace{4}{1}{\pipeCross}
  \PipePlace{4}{2}{\pipeElbow}
  \PipePlace{5}{1}{\pipeElbow}
  \PipeEnd
  \begin{scope}[x=2em,y=2em, yscale=-1, shift={(-6,5)}]
    \node[right] at (5,0.5) {\small $x_1$};
    \node[right] at (5,1.5) {\small $x_3$};
    \node[right] at (5,2.5) {\small $x_2$};
    \node[right] at (5,3.5) {\small $x_5$};
    \node[right] at (5,4.5) {\small $x_4$};
  \end{scope}
\endtikzpicture
\quad 
\tikzpicture
  \PipeStart{5}
  \PipePlace{1}{1}{\pipeCross}
  \PipePlace{1}{2}{\pipeBump}
  \PipePlace{1}{3}{\pipeBump}
  \PipePlace{2}{1}{\pipeCross}
  \PipePlace{2}{2}{\pipeCross}
  \PipePlace{2}{3}{\pipeBump}
  \PipePlace{3}{1}{\pipeBump}
  \PipePlace{3}{2}{\pipeBump}
  \PipePlace{3}{3}{\pipeElbow}
  \PipePlace{4}{1}{\pipeCross}
  \PipePlace{4}{2}{\pipeElbow}
  \PipePlace{5}{1}{\pipeElbow}
  \PipeEnd
  \begin{scope}[x=2em,y=2em, yscale=-1, shift={(-6,5)}]
    \node[right] at (5,0.5) {\small $x_1$};
    \node[right] at (5,1.5) {\small $x_3$};
    \node[right] at (5,2.5) {\small $x_2$};
    \node[right] at (5,3.5) {\small $x_5$};
    \node[right] at (5,4.5) {\small $x_4$};
  \end{scope}
\endtikzpicture
\]
\end{footnotesize}

All reduced BPDs for $w$ are as follows. 

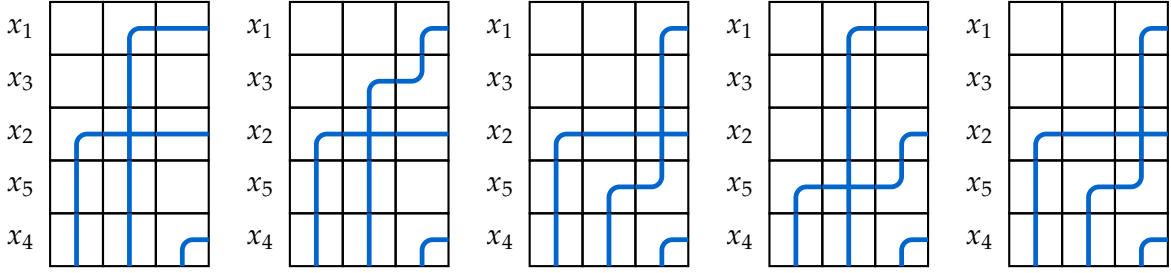
\begin{footnotesize}
\[
\tikzpicture
  \BPDStart{5}{3}
  \BPDPlace{1}{1}{\blank}
  \BPDPlace{1}{2}{\SE}
  \BPDPlace{1}{3}{\hor}
  \BPDPlace{2}{1}{\blank}
  \BPDPlace{2}{2}{\ver}
  \BPDPlace{2}{3}{\blank}
  \BPDPlace{3}{1}{\SE}
  \BPDPlace{3}{2}{\+}
  \BPDPlace{3}{3}{\hor}
  \BPDPlace{4}{1}{\ver}
  \BPDPlace{4}{2}{\ver}
  \BPDPlace{4}{3}{\blank}
  \BPDPlace{5}{1}{\ver}
  \BPDPlace{5}{2}{\ver}
  \BPDPlace{5}{3}{\SE}
  \BPDEnd
  \begin{scope}[x=2em,y=2em, yscale=-1, shift={(-6,5)}]
    \node[right] at (5,0.5) {\small $x_1$};
    \node[right] at (5,1.5) {\small $x_3$};
    \node[right] at (5,2.5) {\small $x_2$};
    \node[right] at (5,3.5) {\small $x_5$};
    \node[right] at (5,4.5) {\small $x_4$};
  \end{scope}
\endtikzpicture
\quad 
\tikzpicture
  \BPDStart{5}{5}
  \BPDPlace{1}{1}{\blank}
  \BPDPlace{1}{2}{\blank}
  \BPDPlace{1}{3}{\SE}
  \BPDPlace{2}{1}{\blank}
  \BPDPlace{2}{2}{\SE}
  \BPDPlace{2}{3}{\NW}
  \BPDPlace{3}{1}{\SE}
  \BPDPlace{3}{2}{\+}
  \BPDPlace{3}{3}{\hor}
  \BPDPlace{4}{1}{\ver}
  \BPDPlace{4}{2}{\ver}
  \BPDPlace{4}{3}{\blank}
  \BPDPlace{5}{1}{\ver}
  \BPDPlace{5}{2}{\ver}
  \BPDPlace{5}{3}{\SE}
  \BPDEnd
 \begin{scope}[x=2em,y=2em, yscale=-1, shift={(-6,5)}]
    \node[right] at (5,0.5) {\small $x_1$};
    \node[right] at (5,1.5) {\small $x_3$};
    \node[right] at (5,2.5) {\small $x_2$};
    \node[right] at (5,3.5) {\small $x_5$};
    \node[right] at (5,4.5) {\small $x_4$};
  \end{scope}
\endtikzpicture \quad
\tikzpicture
  \BPDStart{5}{5}
  \BPDPlace{1}{1}{\blank}
  \BPDPlace{1}{2}{\blank}
  \BPDPlace{1}{3}{\SE}
  \BPDPlace{2}{1}{\blank}
  \BPDPlace{2}{2}{\blank}
  \BPDPlace{2}{3}{\ver}
  \BPDPlace{3}{1}{\SE}
  \BPDPlace{3}{2}{\hor}
  \BPDPlace{3}{3}{\+}
  \BPDPlace{4}{1}{\ver}
  \BPDPlace{4}{2}{\SE}
  \BPDPlace{4}{3}{\NW}
  \BPDPlace{5}{1}{\ver}
  \BPDPlace{5}{2}{\ver}
  \BPDPlace{5}{3}{\SE}
  \BPDEnd
   \begin{scope}[x=2em,y=2em, yscale=-1, shift={(-6,5)}]
    \node[right] at (5,0.5) {\small $x_1$};
    \node[right] at (5,1.5) {\small $x_3$};
    \node[right] at (5,2.5) {\small $x_2$};
    \node[right] at (5,3.5) {\small $x_5$};
    \node[right] at (5,4.5) {\small $x_4$};
  \end{scope}
\endtikzpicture \quad
\tikzpicture
  \BPDStart{5}{5}
  \BPDPlace{1}{1}{\blank}
  \BPDPlace{1}{2}{\SE}
  \BPDPlace{1}{3}{\hor}
  \BPDPlace{2}{1}{\blank}
  \BPDPlace{2}{2}{\ver}
  \BPDPlace{2}{3}{\blank}
  \BPDPlace{3}{1}{\blank}
  \BPDPlace{3}{2}{\ver}
  \BPDPlace{3}{3}{\SE}
  \BPDPlace{4}{1}{\SE}
  \BPDPlace{4}{2}{\+}
  \BPDPlace{4}{3}{\NW}
  \BPDPlace{5}{1}{\ver}
  \BPDPlace{5}{2}{\ver}
  \BPDPlace{5}{3}{\SE}
  \BPDEnd
   \begin{scope}[x=2em,y=2em, yscale=-1, shift={(-6,5)}]
    \node[right] at (5,0.5) {\small $x_1$};
    \node[right] at (5,1.5) {\small $x_3$};
    \node[right] at (5,2.5) {\small $x_2$};
    \node[right] at (5,3.5) {\small $x_5$};
    \node[right] at (5,4.5) {\small $x_4$};
  \end{scope}
\endtikzpicture \quad
\tikzpicture
  \BPDStart{5}{5}
  \BPDPlace{1}{1}{\blank}
  \BPDPlace{1}{2}{\blank}
  \BPDPlace{1}{3}{\SE}
  \BPDPlace{2}{1}{\blank}
  \BPDPlace{2}{2}{\blank}
  \BPDPlace{2}{3}{\ver}
  \BPDPlace{3}{1}{\SE}
  \BPDPlace{3}{2}{\hor}
  \BPDPlace{3}{3}{\+}
  \BPDPlace{4}{1}{\ver}
  \BPDPlace{4}{2}{\SE}
  \BPDPlace{4}{3}{\NW}
  \BPDPlace{5}{1}{\ver}
  \BPDPlace{5}{2}{\ver}
  \BPDPlace{5}{3}{\SE}
  \BPDEnd
  \begin{scope}[x=2em,y=2em, yscale=-1, shift={(-6,5)}]
    \node[right] at (5,0.5) {\small $x_1$};
    \node[right] at (5,1.5) {\small $x_3$};
    \node[right] at (5,2.5) {\small $x_2$};
    \node[right] at (5,3.5) {\small $x_5$};
    \node[right] at (5,4.5) {\small $x_4$};
  \end{scope}
\endtikzpicture
\]
\end{footnotesize}

The monomial weights of the reduced classical pipe dreams above are exactly the monomials appearing in the Schubert polynomial
\[
\mathfrak{S}_w = x_1^2x_2x_3+ x_1^2x_3^2+ x_1x_2x_3^2+x_1^2x_3x_5+x_1x_3^2x_5,
\]
in accordance with \Cref{thm:word-schubert-pd}. In this example there are five reduced classical pipe dreams and five reduced bumpless pipe dreams for $w$, reflecting the five reduced pipe dreams and BPDs for the permutation $u=\std(\conv(w))$ from \Cref{eg:PD-for-24153} and \Cref{eg:BPD-for-24153}.

Next, we display some of the many non-reduced pipe dreams and BPDs for $w$, which contribute the higher-order terms in the corresponding Grothendieck polynomial. Their weights account for the higher-order terms in 
$$\begin{gathered}
    \mathfrak{G}_w = 2 x_1^2 x_2 x_3^2 x_5 - 2 x_1^2 x_2 x_3^2 - 2 x_1^2 x_3^2 x_5 - x_1^2 x_2 x_3 x_5 - x_1 x_2 x_3^2 x_5 \\
    + x_1^2 x_2 x_3 + x_1^2 x_3^2 + x_1 x_2 x_3^2 + x_1^2 x_3 x_5 + x_1 x_3^2 x_5.
\end{gathered}$$

\begin{footnotesize}
    \[\tikzpicture
  \PipeStart{5}
  \PipePlace{1}{1}{\pipeCross}
  \PipePlace{1}{2}{\pipeBump}
  \PipePlace{1}{3}{\pipeCross}
  \PipePlace{2}{1}{\pipeCross}
  \PipePlace{2}{2}{\pipeBump}
  \PipePlace{2}{3}{\pipeCross}
  \PipePlace{3}{1}{\pipeBump}
  \PipePlace{3}{2}{\pipeCross}
  \PipePlace{3}{3}{\pipeElbow}
  \PipePlace{4}{1}{\pipeBump}
  \PipePlace{4}{2}{\pipeElbow}
  \PipePlace{5}{1}{\pipeElbow}
  \PipeEnd
  \begin{scope}[x=2em,y=2em, yscale=-1, shift={(-6,5)}]
    \node[right] at (5,0.5) {\small $x_1$};
    \node[right] at (5,1.5) {\small $x_3$};
    \node[right] at (5,2.5) {\small $x_2$};
    \node[right] at (5,3.5) {\small $x_5$};
    \node[right] at (5,4.5) {\small $x_4$};
  \end{scope}
\endtikzpicture \qquad \tikzpicture
  \PipeStart{5}
  \PipePlace{1}{1}{\pipeCross}
  \PipePlace{1}{2}{\pipeBump}
  \PipePlace{1}{3}{\pipeCross}
  \PipePlace{2}{1}{\pipeCross}
  \PipePlace{2}{2}{\pipeBump}
  \PipePlace{2}{3}{\pipeCross}
  \PipePlace{3}{1}{\pipeBump}
  \PipePlace{3}{2}{\pipeCross}
  \PipePlace{3}{3}{\pipeElbow}
  \PipePlace{4}{1}{\pipeCross}
  \PipePlace{4}{2}{\pipeElbow}
  \PipePlace{5}{1}{\pipeElbow}
  \PipeEnd
  \begin{scope}[x=2em,y=2em, yscale=-1, shift={(-6,5)}]
    \node[right] at (5,0.5) {\small $x_1$};
    \node[right] at (5,1.5) {\small $x_3$};
    \node[right] at (5,2.5) {\small $x_2$};
    \node[right] at (5,3.5) {\small $x_5$};
    \node[right] at (5,4.5) {\small $x_4$};
  \end{scope}
\endtikzpicture \qquad \tikzpicture
  \BPDStart{5}{3}
  \BPDPlace{1}{1}{\blank}
  \BPDPlace{1}{2}{\blank}
  \BPDPlace{1}{3}{\SE}
  \BPDPlace{2}{1}{\blank}
  \BPDPlace{2}{2}{\blank}
  \BPDPlace{2}{3}{\ver}
  \BPDPlace{3}{1}{\blank}
  \BPDPlace{3}{2}{\SE}
  \BPDPlace{3}{3}{\+}
  \BPDPlace{4}{1}{\SE}
  \BPDPlace{4}{2}{\+}
  \BPDPlace{4}{3}{\NW}
  \BPDPlace{5}{1}{\ver}
  \BPDPlace{5}{2}{\ver}
  \BPDPlace{5}{3}{\SE}
  \BPDEnd
  \begin{scope}[x=2em,y=2em, yscale=-1, shift={(-6,5)}]
    \node[right] at (5,0.5) {\small $x_1$};
    \node[right] at (5,1.5) {\small $x_3$};
    \node[right] at (5,2.5) {\small $x_2$};
    \node[right] at (5,3.5) {\small $x_5$};
    \node[right] at (5,4.5) {\small $x_4$};
  \end{scope}
  \endtikzpicture \qquad \raisebox{5em}{$\cdots$}
\]
\end{footnotesize}
\end{example}

An immediate consequence of Definitions~\ref{def:PD-for-word} and~\ref{def:BPD-for-word}, together with the permutation case, is that Schubert and Grothendieck polynomials of a word are again monomial-weight generating functions for (classical or bumpless) pipe dreams.  In the following statements we use the term ``pipe dream'' to mean either a classical pipe dream or a bumpless pipe dream. The proofs apply verbatim in both settings.

\begin{theorem}
\label{thm:word-schubert-pd}
For any word $w\in [k]^n$, the multiset of monomials appearing in the Schubert polynomial $\mathfrak{S}_w(x_1,\dots,x_n)$ coincides with the multiset of monomial weights of reduced pipe dreams for $w$ defined as above.
\end{theorem}

\begin{proof}
Let $u := \std(\conv(w))$.  By \Cref{thm:pipe dreams}, $\mathfrak{S}_u(x_1,\dots,x_n)$ is the weight-generating function for reduced pipe dreams in $\operatorname{PD}(u)$.  By \Cref{lem:rectangular-PD} and \Cref{def:PD-for-word}, restriction to the first $k$ columns and relabeling rows by $\sinv$ induces a weight-preserving bijection between reduced pipe dreams for $u$ and reduced pipe dreams for $w$.

On the polynomial side, the definition of $\mathfrak{S}_w$ is $\mathfrak{S}_w = \sinv\cdot \mathfrak{S}_u$, i.e., $\mathfrak{S}_w$ is obtained from $\mathfrak{S}_u$ by the same permutation of variables.  Hence the multiset of monomials appearing in $\mathfrak{S}_w$ is exactly the multiset of weights of reduced pipe dreams for $w$.
\end{proof}

\begin{theorem}
\label{thm:word-grothendieck-bpd}
For any word $w\in [k]^n$, the multiset of monomials appearing in the Grothendieck polynomial $\mathfrak{G}_w(x_1,\dots,x_n)$ coincides with the multiset of $K$-theoretic monomial weights of non-reduced pipe dreams for $w$ defined as above.
\end{theorem}

\begin{proof}
The argument is similar to the proof of \Cref{thm:word-schubert-pd}, using \Cref{thm:BPD-grothendieck} in place of \Cref{thm:pipe dreams}.  For the permutation $u=\std(\conv(w))$, non-reduced pipe dreams (or bumpless pipe dreams) are in weight-preserving bijection with the monomials of $\mathfrak{G}_u(x_1,\dots,x_n)$.  By the bumpless analogue of \Cref{lem:rectangular-PD} together with \Cref{def:BPD-for-word}, restriction to the first $k$ columns and relabeling by $\sinv$ preserves weights and induces a bijection between non-reduced pipe dreams for $u$ and for $w$.  On the polynomial side, we again have $\mathfrak{G}_w = \sinv\cdot \mathfrak{G}_u$, so the multisets of monomials agree.
\end{proof}

For Fubini words $w\in\Fubini$, the Schubert and Grothendieck polynomials $\mathfrak S_w$ and $\mathfrak G_w$ index the canonical bases of $H^\bullet(X_{n,k})$ and $K_0(X_{n,k})$.  Theorems~\ref{thm:word-schubert-pd} and~\ref{thm:word-grothendieck-bpd} therefore provide pipe-dream and bumpless-pipe-dream descriptions of the basis elements of these rings, extending the classical permutation story to the setting of spanning line configurations and the generalized coinvariant algebra.  It would be interesting to compare these word-level models with back-stable Schubert and Grothendieck polynomials in the sense of Lam--Lee--Shimozono \cite{LLS}, and to investigate whether similar stabilization and symmetry phenomena occur in this broader context.
\section{Open problems and further directions}
\label{sec:open-problems}

The results of this paper place the variety $X_{n,k}$ of spanning line configurations among the class of smooth cellular varieties whose cohomology, Chow ring, and Grothendieck group are all abstractly isomorphic to a combinatorially defined ring.  We conclude by recording several directions where further progress would be interesting.

Recall Question~\ref{qs:when-are-K0-CH-iso}, which asks for a characterization of schemes $X$ for which the rings $K_0(X)$ and $CH^\bullet(X)$ are abstractly isomorphic over $\ZZ$. Well-known examples include projective spaces, Grassmannians, and other partial flag varieties. In addition, Larson-Li-Payne-Proudfoot in \cite{LLPP24WonderfulVarieties} proved that the $K_0$-rings of \textit{wonderful varieties} in the sense of Concini-Procesi \cite{ConciniProcesi95Wonderful} are isomorphic to their Chow rings over $\ZZ$. Our main theorem shows that $X_{n,k}$ provides a new kind of example. It would be valuable to understand how far this phenomenon extends beyond all of the above examples.

\begin{question}[Integral $K_0$ versus Chow]
\label{qs:when-are-K0-CH-iso-open}
Find natural necessary and sufficient conditions on a scheme $X$ such that there exists an isomorphism of ungraded rings
\[
K_0(X) \;\cong\; CH^\bullet(X)
\]
over the integers.  In particular, is there a conceptual explanation for why projective spaces, Grassmannians, full flag varieties, and $X_{n,k}$ all satisfy this property, while certain smooth cellular varieties such as $\PP(\cO_{\PP^2}(2))$ do not?
\end{question}

In another direction, Chou--Matsumura--Rhoades \cite{chou2024equivariantcohomologygrassmannianspanning} gave a presentation of the $T$-equivariant cohomology of $X_{n,k}$, where $T$ is the natural torus acting on $(\PP^{k-1})^n$.  General principles suggest that the $T$-equivariant Chow ring of $X_{n,k}$ should agree with its equivariant cohomology, but the equivariant $K$-theory is subtler.

\begin{question}[Equivariant $K$-theory]
\label{qs:eq-K-theory}
Determine the structure of the $T$-equivariant $K$-theory $K^T_0(X_{n,k})$.  In particular, does $K^T_0(X_{n,k})$ admit a presentation analogous to that of the equivariant cohomology in \cite{chou2024equivariantcohomologygrassmannianspanning}, and to what extent is $K^T_0(X_{n,k})$ determined by $H^\bullet_T(X_{n,k})$?
\end{question}

Our pipe dream and bumpless pipe dream models for Fubini words extend many features of the classical permutation case.  One important aspect of the permutation story is the interpretation of both classical and bumpless pipe dreams as wiring diagrams for permutations.  It would be interesting to understand whether there is an equally natural “wiring” picture in the Fubini-word setting.

\begin{question}[Wiring diagrams for words]
\label{qs:wiring}
Is there a notion of wiring diagram for Fubini words that corresponds naturally to classical and/or bumpless pipe dreams of a word $w\in [k]^n$?  Can such diagrams be used to give more direct bijections with reduced decompositions or other combinatorial data associated to $w$?
\end{question}

Finally, Weigandt \cite{weigandt2025changingbasespipedream} obtained change-of-basis formulas for the coinvariant algebra $R_n$ between Schubert and Grothendieck bases, expressed in terms of pipe dream combinatorics.  Since $\Rnk$ admits both Schubert and Grothendieck bases indexed by Fubini words, one expects analogous phenomena in the generalized setting.

\begin{question}[Change of basis in $R_{n,k}$]
\label{qs:change-of-basis}
Are there combinatorial change-of-basis formulas between the Schubert and Grothendieck bases of the generalized coinvariant algebra $R_{n,k}$, expressed in terms of classical or bumpless pipe dreams for words?  Can such formulas be related to the geometry of $X_{n,k}$ and its Pawlowski--Rhoades strata?
\end{question}

We hope that the interplay between geometry, $K$-theory, and pipe dream combinatorics developed here will be a useful starting point for addressing these questions.


%
%
%

\newpage
\appendix

\section{Algorithms for Pawlowski-Rhoades Cells}
\label{apx:algorithms}

\subsection{Construction of the Pattern Matrix}
\begin{center}
    \begin{algorithm}[h!]
\DontPrintSemicolon
\SetKwInOut{Input}{Input}
\SetKwInOut{Output}{Output}

\caption{\textsc{Construction of} \(\PM(w)\). }
\label{alg:pattern-matrix}

\begin{multicols}{2}
\Input{A word \( w = w_1 w_2 \dots w_n \in \Words \)}
\Output{A \(k \times n\) matrix \(\PM(w)\) with entries in \(\{0,1,\star\}\)}

\BlankLine
\textbf{Initialize the matrix:}\\
\Indp
$\PM(w) \gets \big(0\big)_{k\times n}$
\Indm

\BlankLine
\textbf{Initialize a list of first occurrences of letters:}\\
\Indp
$F\gets \underset{k}{\underbrace{(0, \dots, 0)}}$
\For{$i = 1$ \KwTo $k$}{
\uIf{$w$ contains letter $i$}{
$F_i\gets $ position of first occurrence of $i$ in $w$
}
\Else{$F_i\gets\infty$}
}
\Indm

\BlankLine

\For{$j = 1$ \KwTo $n$}{
\textbf{Create pivot 1's:} \(\PM(w)_{w_j,j} \gets 1\) \\
\BlankLine
\textbf{Create other entries:}\\
\For{$i = 1$ \KwTo $k$}{
\uIf{$j\in \operatorname{in}(w)$, $i<w_j$, \textbf{and} $F_i<j$}{
    \(\PM(w)_{i,j} \gets \star\)
}
\uIf{$j\in \operatorname{re}(w) \textbf{ and } F_i<F_{j}$}{
                \(\PM(w)_{i,j} \gets \star\)
            }
       }
}

\BlankLine
\Return \(\PM(w)\)
\end{multicols}
\end{algorithm}
\vspace{-2em}
\end{center}

\subsection{The Reduction Algorithm}

Given $p = [\ell_\bullet]\in \Pnk$, let $d_i = \dim \ell_1 + \ell_2 +\cdots + \ell_i$ be the dimension of the span of the first $i$ lines in $\ell_\bullet$ for all $1\leq i\leq n$ and set $d_0 = 0$ by convention. The $n+1$-tuple $(d_0, d_1, \dots, d_n)$ satisfies $d_{i} = d_{i-1}$ or $d_{i-1}+1$ for all $1\leq i\leq n$. An index $i$ at which $d_i = d_{i-1}+1$ is called an \textit{initial position}.

\vspace{1em}

\noindent\hrule
\vspace{0.4ex}
\hrule
\medskip
\textbf{Algorithm 2:} \textsc{Reduction Algorithm} \label{alg:reduction-alg}.
\medskip
\hrule
\vspace{0.4ex}
\hrule
\medskip
\textbf{Input:} $k\times n$-matrix $m=\left(m_{i, j}\right)\in\Nonzerocol$
\medskip

\textbf{Output: } $(\operatorname{R}(m), \operatorname{word}(m))$ where $\operatorname{R}(m)$ is a $k\times n$ canonical coset representative for $UmT$ and $\operatorname{word}(m)\in \Words$.
\medskip

\textbf{Initialize:} Set $m^\prime\gets m$.

\begin{enumerate}[(1)]
    \item For column 1, let $i$ be the minimal row index such that $m^\prime_{i, 1} \neq 0$. By performing a minimal number of elementary row operations, one can find $u\in U$ such that all entries of $um^\prime$ below $(i, 1)$ are zero. The diagonal matrix $t = \operatorname{diag}(1, \dots, \frac{1}{m^\prime_{i,1}}, \dots, 1)\in T$ is such that $(um^\prime t)_{i,1} = 1$. Update $m^\prime \gets um^\prime t$. Set $w_1 \gets i$.
    \item For each subsequent column $2 \leq j \leq n$ of $m^\prime$: 
\begin{enumerate}
    \item If there exists a row index $i\in [k]-\{w_1, \dots, w_{j-1}\}$ for which $m^\prime_{i, j} \neq 0$, then choose the minimal such $i$. Update $u\in U$ as above such that all entries of $um^\prime$ below $(i, j)$ are zero. Update $t \gets \operatorname{diag}(1, \dots, \frac{1}{m^\prime_{i,j}}, \dots, 1)\in T$ such that $(um^\prime t)_{i,j} = 1$. Update $m^\prime \gets um^\prime t$. Set $w_j \gets i$.
    \item If no such $i$ exists, then let $\left(i_1< i_2< \ldots< i_r\right)$ be the initial letters of $w=w_1, \dots, w_{j-1}$. Choose maximal $i \in \left(i_1< i_2< \ldots< i_r\right)$ such that $m^\prime_{i, j} \neq 0$. Update 
    
    $t \gets \operatorname{diag}(1, \dots, \frac{1}{m^\prime_{i,j}}, \dots, 1)\in T$ so that $(m^\prime t)_{i, j}=1$. Update $m^\prime \gets m^\prime t$.  Set $w_j \gets i$.
\end{enumerate}
\end{enumerate}

\textbf{Return:} $(m^\prime, w)$.
\medskip
\hrule

\section*{Acknowledgments}
The author is grateful to Sara Billey and Jarod Alper for thought-provoking conversations. The author thanks Brendan Pawlowski and Brendon Rhoades for ideas and clarifications. The author thanks Elena Hafner and Anna Weigandt for helpful comments.

\bibliographystyle{amsalpha-ac}
\bibliography{sample_2025}
\end{document}